%% file: main.tex
\documentclass[12pt,pdftex,noinfoline]{imsart}

\RequirePackage[OT1]{fontenc}
\usepackage{amsthm,amsmath,amsfonts,natbib,mathtools,amssymb}
\RequirePackage{hypernat}
\usepackage[ruled,section]{algorithm}
\usepackage{algorithmic}
\usepackage{graphicx}
\usepackage{verbatim}
\usepackage{times}
\usepackage[T1]{fontenc}
\usepackage[text={6.0in,8.6in},centering]{geometry}
\usepackage{graphicx}
\usepackage{hyperref}
\usepackage{yhmath}
\usepackage[usenames,dvipsnames]{xcolor}
\definecolor{shadecolor}{gray}{0.9}
\definecolor{shadecolor}{gray}{0.9}
\hypersetup{citecolor=MidnightBlue}
\hypersetup{linkcolor=MidnightBlue}
\hypersetup{urlcolor=MidnightBlue}
\usepackage{enumerate}
\usepackage{fullpage}
\usepackage{tikz}
\usepackage{dsfont}
\usepackage{accents}
\usepackage{amssymb}
\usepackage{amsthm}
\usepackage{txfonts}
\usepackage{bbm}
\usepackage{subcaption}
\usepackage{wrapfig}
\usepackage{afterpage}

\newcommand{\E}{\operatorname{\mathbb E}}

\DeclareMathOperator*{\argmin}{\arg\min}

\numberwithin{equation}{section}
\newtheorem{theorem}{Theorem}[section]
\newtheorem{lemma}{Lemma}[section]
\newtheorem{corollary}{Corollary}[section]

\theoremstyle{remark}

\newtheorem{remark}{Remark}[section]

\def\given{\,|\,}
\def\P{\mathbb{P}}
\def\E{\mathbb{E}}
\def\reals{\mathbb{R}}
\def\argmin{\mathop{\text{arg\,min}}}
\def\ones{\mathds{1}}

\let\hat\widehat
\let\what\widehat
\let\tilde\widetilde

\begin{document}

\setlength{\parskip}{0.5em}

\begin{frontmatter}
\title{Fair Quantile Regression}
\runtitle{Fair Quantile Regression}
%\affiliation{Department of Statistics, University of Pennsylvania}
%\affil[**]{Department of Statistics and Data Science, Yale University}
\begin{aug}
\vskip10pt
\author{\fnms{Dana} \snm{Yang}\ead[label=e1]{xiaoqian.yang@yale.edu}}
\,
\author{\fnms{John} \snm{Lafferty}\ead[label=e2]{john.lafferty@yale.edu}}
\,
\author{\fnms{David} \snm{Pollard}\ead[label=e3]{david.pollard@yale.edu}}
\vskip10pt
\address{
\begin{tabular}{c}
Department of Statistics and Data Science\\
Yale University
\end{tabular}
\\[10pt]
\today\\[5pt]
\vskip10pt
}
\end{aug}

\begin{abstract}
Quantile regression is a tool for learning conditional
distributions. In this paper we study quantile regression in the
setting where a protected attribute is unavailable when fitting the
model. This can lead to ``unfair'' quantile estimators for which the
effective quantiles are very different for the subpopulations defined
by the protected attribute.  We propose a procedure for adjusting the estimator on a
heldout sample where the protected attribute is available. The main result
of the paper is an empirical process analysis showing
that the adjustment leads to a fair estimator for which the target
quantiles are brought into balance, in a statistical sense that we
call $\sqrt{n}$-fairness. We illustrate the ideas and adjustment
procedure on a dataset of $200{,}000$ live births, where the objective
is to characterize the dependence of the birth weights of the babies
on demographic attributes of the birth mother; the protected attribute
is the mother's race.
\end{abstract}
\end{frontmatter}

\input{intro}

\input{background}
\input{results}
\input{training}
\input{proof2}
\input{experiments}
\input{birth}
\input{discussion}
\input{proof}
\input{ack}

\bibliographystyle{apalike}
\bibliography{fairness}

%\end{document}
\newpage

\end{document}

%% file: intro.tex
% !TEX root = main.tex

\section{Introduction}

Recent research on fairness has formulated interesting new perspectives on
machine learning methodologies and their deployment, through work
on definitions, axiomatic characterizations, case studies, and
algorithms
\citep{hardt2016equality,dwork12,kleinberg2016inherent,chouldechova2017fair,woodworth2017learning}.

Much of the work on fairness in machine learning has
been focused on classification, although the influential paper
of \cite{hardt2016equality} considers general frameworks that include regression.
Just as the mean gives a coarse summary of
a distribution, the regression curve gives a rough summary
of a family of conditional distributions \citep{MostellerTukey77}.
Quantile regression targets a more complete
understanding of the dependence between a response variable and a
collection of explanatory variables.

Given a conditional distribution $F_X(y) = \P(Y\leq y \given X)$,
the quantile function $q_\tau(X)$ is characterized by $F_X(q_\tau(X))
= \tau$, or $q_\tau(X) = F_X^{-1}(\tau) = \inf \{y: F_X(y) \geq \tau\}$.
We consider the setting where an estimate $\what q_\tau(X)$ is formed
using a training set $\{(X_i, Y_i)\}$ for which a protected
attribute $A$ is unavailable. The estimate $\what q_\tau(X)$ will often give
quantiles that are far from $\tau$, when conditioned on
the protected variable. We study methods that adjust the estimator using a heldout
sample for which the protected attribute $A$ is observed.

As example, to be developed at length below, consider forecasting
the birth weight of a baby as a function of
the demographics and personal history of the birth mother, including
her prenatal care, smoking history, and educational background. As
will be seen, when the race of mother is excluded, the quantile
function may be very inaccurate, particularly at the lower quantiles
$\tau < 0.2$ corresponding to low birth weights. If used
as a basis for medical advice, such inaccurate forecasts could conceivably have health consequences for the
mother and infant. It would be important to adjust the estimates if the race of the mother became available.

In this paper we study the simple procedure that adjusts an initial estimate
$\what q_\tau(X)$ by adding $\what \mu_\tau A + \hat\nu_\tau$, by carrying out a quantile regression of $Y-\what q_\tau(X)$ onto $A$.
We show that this leads to an estimate
$\tilde{q}_\tau(X,A)  = \what q_\tau(X) + \what \mu_\tau A + \hat\nu_\tau$ for which the conditional quantiles are
close to the target level $\tau$ for both subpopulations $A=1$ and $A=0$. This result follows
from an empirical process analysis that exploits the special dual structure of quantile regression as a linear program.
The main technical result of our paper is that our adjustment procedure is
$\sqrt{n}$-fair at the population level. Roughly speaking, this means
that the effective quantiles for the two subpopulations agree, up to a
stochastic error that decays at a parametric $1/\sqrt{n}$ rate. We establish this result using empirical process techniques that generalize to more general types of attributes, not just binary.

In the following section we provide
technical background on quantile regression, including its formulation
in terms of linear programming, the dual program, and methods for
inference.  We also provide background on notions of fairness that are related to this work and give our definition of fairness. In Section~\ref{sec:results} we formally state the methods and results.
  The key steps in the proof are given in Section~\ref{sec:proof}. We illustrate these results on synthetic data and birth weight data in Section \ref{sec:experiments}.
We finish with a discussion of the results and possible directions for future work.
Full proofs of the technical results are provided in Section~\ref{sec:pop}

%% file: background.tex
% !TEX root = main.tex

\vskip20pt

\section{Background}
\label{sec:background}

In this section we review the essentials of quantile
regression that will be relevant to our analysis. We also
briefly discuss definitions of fairness.

\subsection{Linear programming formulation}

The formulation of quantile estimates as solutions to linear programs starts with the ``check'' or ``hockey stick'' function $\rho_\tau(u)$ defined by $\rho_\tau(u)= (\tau-1) u\mathds{1}\{u\leq 0\}+ \tau u\mathds{1}\{u>0\}$.

For the median, $\rho_{1/2}(u) = \frac{1}{2} |u|$. If $Y\sim F$ is a random variable,
define $\hat\alpha(\tau)$ as the solution to the optimization $\hat\alpha(\tau) = \argmin_a \E \rho_\tau(Y-a)$. Then the stationary condition is seen to be
\begin{equation*}
    0 = (\tau-1) \int_{-\infty}^\alpha dF(u) + \tau \int_\alpha^\infty dF(u) = (\tau-1) F(\alpha) + \tau (1-F(\alpha)),
\end{equation*}
from which we conclude $\hat\alpha(\tau) = F^{-1}(\tau)$ is the $\tau$-quantile of $F$. Similarly the conditional quantile of $Y$ given random variable $X\in\mathbb{R}^p$ can be written as the solution to the optimization
$
    q_\tau(x) = \argmin_q \E\left( \rho_\tau(Y - q) \given X=x\right).
$
For a linear estimator $q_\tau(X) = X^T \hat{\beta}_\tau$, minimizing the
empirical check function loss leads to a convex optimization $\hat{\beta}_\tau=\argmin_{\beta} \sum_{i\leq n}\rho_\tau (Y_i-X_i^T\beta)$. Dividing the residual $Y_i-X_i^T\beta$ into positive part $u_i$ and negative part $v_i$ yields the linear program
\[
    \min_{u,v\in\mathbb{R}^n, \beta\in\mathbb{R}^p} \;  \tau \ones^T u + (1-\tau)\ones^T v,\;\;\;
    \text{such that} \;  Y = X\beta + u - v ,\;
                         u \geq 0, \; v\geq 0.
\]
The dual linear program is then formulated as
\begin{equation}
    \label{eq:dual}
    \max_{b} \;  Y^T b \;\;\;
    \text{such that} \;  X^T b = (1-\tau) X^T \ones,\;
    b \in [0,1]^n.
\end{equation}
When $n>p$, the primal
solution is obtained from a set of $p$ observations $X_h\in \reals^{p\times p}$ for
which the residuals are exactly zero, through
the correspondence
$
\hat\beta_\tau = X_h^{-1} Y_h.
$
The dual variables $\hat b_\tau\in [0,1]^n$, also known as regression rank scores, play the role of ranks. In particular, the quantity $\int_0^1 \hat b_{\tau,i} d\tau$ can be interpreted as the quantile
at which $Y_i$ lies for the conditional distribution of $Y$ given
$X_i$ \citep{gutenbrunner1992regression}. As seen below,
the stochastic process $\hat b_\tau$ plays an important role in fairness and inference for quantile regression.

\subsection{Notions of fairness}
\label{sec:fair.def}

\cite{hardt2016equality} introduce the notion of
{\it equalized odds} to assess fairness of
classifiers. Suppose a classifier $\widehat{Y}$ serves to estimate
some unobserved binary outcome variable $Y$. Then the estimator is
said to satisfy the equalized odds property with respect to a
protected attribute $A$ if
\begin{equation}\label{eo.def}
\widehat{Y}\Perp A \given Y.
\end{equation}
This fairness property requires that the true positive
rates~$\mathbb{P}\{\widehat{Y}=1 \given Y=1, A\}$ and the false
positive rates $\mathbb{P}\{\widehat{Y}=1 \given Y=0, A\}$ are
constant functions of~$A$. In other words,
$\widehat{Y}$ has the same proportion of type-I and type-II errors
across the subpopulations determined by the different values of $A$.

This could be extended to a related notion of fairness for quantile
regression estimators. Denote the
true conditional quantiles for outcome $Y$ given attributes $X$ as
$q_\tau(X)$. Analogous to the definition of equalized odds in~\eqref{eo.def}, we
would call a quantile estimator $\widehat{q}_\tau(X)$ fair if
\begin{equation}
\mathds{1} \left\{Y>\widehat{q}_\tau(X)\right\}\Perp A \given \mathds{1}\left\{Y>q_\tau(X)\right\}.
\end{equation}
Conditioned on the event  $\left\{Y
\leq q_\tau(X)\right\}$, we say that
$\left\{Y > \widehat{q}_\tau(X)\right\}$ is a false positive.
Conditioned on the complementary event $\left\{Y>q_\tau(X)\right\}$, we
say that $\left\{Y \leq \widehat{q}_\tau(X)\right\}$ is a false
negative. Thus, an estimator is fair if the false positive and false
negative rates do not depend on the protected attribute $A$.

The notion of fairness that we focus on in this paper is a natural
one. Considering binary $A$, we ask if the average quantiles conditional on the protected
attribute agree for $A=0$ and $A=1$. More precisely, define the effective quantiles as
\begin{equation}
\label{eq:effective.quantile}
\hat{\tau}_a=\mathbb{P}\left\{Y\leq \hat{q}_\tau(X)\given A=a\right\},\;\;\; a = 0,1.
\end{equation}
We say that the estimator $\hat q_\tau$ is fair if $\hat{\tau}_0=\hat{\tau}_1$. 
Typically when $\hat{q}_\tau$ is trained on a sample of size $n$, exact equality is too strong to ask for.
If the estimators are accurate, each of the effective quantiles should be approximately
$\tau$, up to stochastic error that decays at rate $1/\sqrt{n}$.
We say $\hat{q}_\tau$ is $\sqrt{n}$-fair if $\hat{\tau}_0=\hat{\tau}_1+O_p(1/\sqrt{n})$.
As shall be seen,
this fairness property follows from the linear programming
formulation when $A$ is included in the regression. As seen from the
birth weight example in Section~\ref{sec:birth},
if $A$ is not available at training time, the quantiles can be
severely under- or over-estimated for a subpopulation.
This formulation of fairness is closely related
to calibration by group, and demographic parity
\citep{kleinberg2016inherent,hardt2016equality,chouldechova2017fair}.
An advantage of this fairness definition is that it can be
evaluated empirically, and does not require a correctly specified model.

%% file: results.tex
% !TEX root = main.tex

\section{Method and Results}
\label{sec:results}

With samples $(A_i,X_i,Y_i)$ drawn i.i.d. from some joint distribution $F$ on $\mathbb{R}\times \mathbb{R}^{p}\times \mathbb{R}$, consider the problem of estimating the conditional quantile $q_\tau(y \given a,x)$. Let $E_F$ denote the expected value operator under $F$, or $E_F f=\int f(a,x,y)\,dF(a,x,y)$. Similarly define the probability operator under $F$ as $P_F$.

Evaluate the level of fairness of an estimator $\hat{q}_\tau(a,x)$ with
\[
\text{Cov}_F\left(a,\mathds{1}\left\{y> \hat{q}_\tau(a,x)\right\}\right)= E_F(a-E_F a)\left(\mathds{1}\left\{y> \hat{q}_\tau(a,x)\right\}-P_F\left\{y> \hat{q}_\tau(a,x)\right\}\right).
\]
An estimator with a smaller $|\text{Cov}_F(a,\mathds{1}\{y> \hat{q}_\tau\})|$ is considered more fair. This measurement of fairness generalizes the notion of balanced effective quantiles described in section~\ref{sec:fair.def}. Note that when the protected attribute is binary, $\text{Cov}_F(a,\mathds{1}\{y> \hat{q}_\tau\})=0$ is equivalent to $\hat{\tau}_0=\hat{\tau}_1$ for $\hat{\tau}$ defined in~\eqref{eq:effective.quantile}.

From an initial estimator $\hat{q}_\tau$ that is potentially unfair, we propose the following correction procedure.
\vskip10pt
\framebox{
\parbox{\dimexpr\linewidth-10\fboxsep-2\fboxrule}
{\itshape%
On a training set of size $n$, compute $R_i=Y_i-\hat{q}_\tau(A_i,X_i)$ and run quantile regression of $R$ on $A$ at level $\tau$. Obtain regression slope $\hat{\mu}_\tau$ and intercept $\hat{\nu}_\tau$. Define correction $\tilde{q}_\tau(a,x)=\hat{q}_\tau(a,x)+\hat{\mu}_\tau a+\hat{\nu}_\tau$.
}}
\vskip10pt

We show that this estimator $\tilde{q}_\tau$ will
satisfy the following:
\begin{enumerate}
\item Faithful: $P_F\{y>\tilde{q}_\tau\}\approx 1-\tau$;% Is faithful the right terminology?
\item Fair: $\text{Cov}_F(a,\mathds{1}\{y>\tilde{q}_\tau\})\approx 0$;
\item Reduced risk: It almost always improves the fit of $\hat{q}_\tau$.
\end{enumerate}
Theorem~\ref{thm:fair} and Theorem~\ref{thm:risk} contain the precise statements of our claims.

\begin{theorem}[Faithfulness and fairness]\label{thm:fair}
Suppose $(A_i,X_i,Y_i)\stackrel{i.i.d.}{\sim} F$, and $A_i-\mathbb{E}A_i$ has finite second moment. Then the corrected estimator $\tilde{q}_\tau$ satisfies
\begin{align}
\label{eq:faithful.rate}
\sup_\tau\bigl|P_F\{y>\tilde{q}_\tau(a,x)\}-(1-\tau)\bigr| = & O_p\left(1/\sqrt{n}\right),\;\;\;\;\;\text{ and }\\
\label{eq:fair.rate}
\sup_\tau\bigl|\text{Cov}_F\left(a,\mathds{1}\{y>\tilde{q}_\tau(a,x)\}\right)\bigr|= & O_p\left(1/\sqrt{n}\right).
\end{align}
Furthermore, there exist positive constants $C,C_1,C_2$ such that $\forall t>C_1/\sqrt{n}$,
\begin{equation}\label{eq:faithful}
\mathbb{P}\left\{\sup_\tau\bigl|P_F\{y>\tilde{q}_\tau(a,x)\}-(1-\tau)\bigr|>t+p/n\right\}\leq C\exp\left(-C_2 nt^2\right).
\end{equation}
Under the stronger assumption that the distribution of $A_i-\mathbb{E}A_i$ is sub-Gaussian, there exist positive constants $C,C_1,C_2,C_3,C_4$ such that $\forall t>C_1/\sqrt{n}$,
\begin{equation}\label{eq:fair}
\mathbb{P}\left\{\sup_\tau\bigl|\text{Cov}_F\left(a,\mathds{1}\{y>\tilde{q}_\tau(a,x)\}\right)\bigr|>t\right\}\leq
C\left(\exp\left(-C_2 nt^2\right)+\exp\left(-C_3 \sqrt{n}\right)+n\exp\left(-C_4 n^2t^2\right)\right).
\end{equation}
\end{theorem}

The following corollary for binary protected attributes is an easy consequence of~\eqref{eq:faithful.rate} and~\eqref{eq:fair.rate}.
\begin{corollary}
\label{cor:fair}
If $A$ is binary, then the correction procedure gives balanced effective quantiles:
\[
\hat{\tau}_0=\tau+O_p(1/\sqrt{n}),\;\;\; \hat{\tau}_1=\tau+O_p(1/\sqrt{n}).
\]
\end{corollary}

\begin{remark}
By modifying the proof of Theorem~\ref{thm:fair} slightly, Corollary~\ref{cor:fair} can be extended to the case where $A$ is categorical with $K$ categories. In this case the correction procedure needs to be adjusted accordingly. Instead of regressing $R$ on $A$, regress $R$ on the span of the indicators $\{A=k\}$ for $k=1,..., K-1$, leaving one category out to avoid collinearity. The corrected estimators will satisfy $\hat{\tau}_k=\tau+O_p(1/\sqrt{n})$ for all categories $k=1,..., K$.
\end{remark}

Define $\mathcal{R}(\cdot)=E_F\rho_\tau (y-\cdot)$ as the risk function, where $\rho_\tau(u)=\tau u\mathds{1}\{u>0\}+(1-\tau) u\mathds{1}\{u\leq 0\}$.
\begin{theorem}[Risk quantification]\label{thm:risk}
The adjustment procedure $\tilde{q}_\tau(A,X)$ satisfies
\[
\mathcal{R}(\tilde{q}_\tau)\leq \inf_{\mu,\nu\in\mathbb{R}}\mathcal{R}(\hat{q}_\tau+\mu A + \nu) +O_p(1/\sqrt{n}).
\]
\end{theorem}

We note that in the different setting where $A$ is a treatment rather than an observational variable,
it is of interest to obtain an unbiased estimate of the treatment effect $\mu_\tau$.
In this case a simple alternative approach is the so-called ``double machine learning'' procedure
by \cite{chernozhukov2016double}; in the quantile regression setting
this would regress the residual onto the transformed attribute $A - \what A$ where
$\what A = \what A(X)$ is a predictive model
of $A$ in terms of $X$.

%% file: training.tex
% !TEX root = main.tex

\section{Fairness on the training set}
\label{sec:training}

When a set of regression coefficients $\hat{\beta}_\tau$ is obtained by running quantile regression of $Y\in\mathbb{R}^n$ on a design matrix $X\in\mathbb{R}^{n\times p}$, the estimated conditional quantiles on the training set $\hat{q}_\tau(X)=X^T\hat{\beta}_\tau$ are always ``fair" with respect to any binary covariate that enters the regression. Namely, if a binary attribute $X_j$ is included in the quantile regression, then no matter what other attributes are regressed upon, on the training set the outcome $Y$ will lie above the estimated conditional quantile for approximately a proportion $1-\tau$, for each of the two subpopulations $X_{j}=0$ and $X_j=1$. This phenomenon naturally arises from the mathematics behind quantile regression. This section explains this property, and lays some groundwork for the out-of-training-set analysis of the following section.

We claim that for any binary attribute $X_j$, the empirical effective quantiles are balanced:
\begin{equation}
\label{eq:effect.quant}
\mathbb{P}_n\{Y>\hat{q}_\tau(X) \given X_j=0\} \approx \mathbb{P}_n\{Y>\hat{q}_\tau(X) \given X_j=1\} \approx 1-\tau,
\end{equation}
where $\mathbb{P}_n$ denotes the empirical probability measure on the training set.

To see why~\eqref{eq:effect.quant} holds, consider the dual of the quantile regression LP \eqref{eq:dual}.
This optimization has Lagrangian
\[
\mathcal{L}(b,\beta)= -Y^Tb+\beta^T(X^T b-(1-\tau)X^T\mathds{1})
%= & -(Y-X\beta)^T b-(1-\tau)\beta^T X^T\mathds{1}\\
=  -\sum_{i\leq n}(Y_i-X_i^T\beta)b_i-(1-\tau)\sum_{i\leq n}X_i^T\beta.
\]
For fixed $\beta$, the vector $b\in [0,1]^n$ minimizing the Lagrangian tends to lie on the ``corners'' of the $n$-dimensional cube, with many of its coordinates taking value either 0 or 1 depending on the sign of $Y_i-X_i^T\beta$. We thus arrive at a characterization for $\hat{b}_\tau$, the solution to the dual program. For $i$ such that $Y_i\neq X_i^T\hat{\beta}_\tau$,
$
\hat{b}_{\tau,i}=\mathds{1}\{Y_i>X_i^T\hat{\beta}_\tau\}.
$
For $i$ such that $Y_i=X_i^T\hat{\beta}_\tau$, the values $\hat{b}_{\tau,i}$ are solutions to the linear system that makes~\eqref{eq:dual} hold. But with $p$ covariates and $n> p$, such equality will typically only occur at most $p$ out of $n$ terms. For large $n$, these only enter the analysis as lower order terms. Excluding these points, the equality constraint in~\eqref{eq:dual} translates to
\begin{equation}\label{eq:effect.quant1}
\sum_{i}X_{ij}\mathds{1}\{Y_i>X_i^T\hat{\beta}_\tau\}=(1-\tau)\sum_{i}X_{ij}\quad\text{for all }j.
\end{equation}
Assuming that the intercept is included as one of the regressors, the above implies that
\begin{equation*}
\frac{1}{n}\sum_{i} \mathds{1}\{Y_i>X_i^T\hat{\beta}_\tau\}=1-\tau,
\end{equation*}
which together with~\eqref{eq:effect.quant1}, implies balanced effective quantiles for binary $X_{\cdot j}$.
In particular, if the protected binary variable $A$ is included in the regression, the resulting model will be fair on the training data,
in the sense that the quantiles for the subpopulations $A=0$ and $A=1$ will be approximately equal, and at the targeted level $\tau$.

This insight gives reason to believe that the quantile regression coefficients, when evaluated on an independent heldout set, should still produce conditional quantile estimates that are what we are calling $\sqrt{n}$-fair. In the following section we establish $\sqrt{n}$-fairness for our proposed adjustment procedure. This requires us to again exploit the connection between the regression coefficients and the fairness measurements formed by the duality of the two linear programs.

%% file: proof2.tex
% !TEX root = main.tex

\vskip20pt
\section{Proof Techniques}\label{sec:proof}

In this section we outline the main steps in the proofs of Theorems~\ref{thm:fair} and~\ref{thm:risk} on the fairness and risk properties of the adjustment procedure. We defer the details of the proofs to Section~\ref{sec:pop}.

We first establish some necessary notation.
From the construction of $\tilde{q}_\tau$, the event $\{y>\tilde{q}_\tau\}$ is equivalent to $\{r>\hat{\mu}_\tau a+\hat{\nu}_\tau\}$ for $r=y-\hat{q}_\tau(a,x)$, which calls for analysis of stochastic processes of the following form. For $d\in\mathbb{R}^n$, let
\[
W_d(\mu,\nu)=\frac{1}{n}\sum_{i\leq n}d_i\left\{R_i>\mu A_i+\nu\right\}.
\]
Let $\overline{W}_d(\mu,\nu)=\mathbb{E}W_d(\mu,\nu)$. It is easy to check that
\[
E_F(\{y>\tilde{q}_\tau\})=\overline{W}_{\mathbbm{1}}\left(\hat{\mu}_\tau,\hat{\nu}_\tau\right),\;\;\;\;\;\text{Cov}_F(a,\{y>\tilde{q}_\tau\})=\overline{W}_{A-\mathbb{E}A}\left(\hat{\mu}_\tau,\hat{\nu}_\tau\right).
\]
The following lemma is essential for establishing concentration results of the $W$ processes around $\overline{W}$, on which the proofs of Theorem~\ref{thm:fair} and Theorem~\ref{thm:risk} heavily rely.

\begin{lemma}\label{lmm:emp}
Suppose $\mathcal{F}$ is a countable family of real functions on $\mathcal{X}$ and $P$ is some probability measure on $\mathcal{X}$. Let $X_1, ..., X_n \stackrel{i.i.d.}{\sim} P$. If
\begin{enumerate}
\item there exists $F:\mathcal{X}\rightarrow \mathbb{R}$ such that $|f(x)|\leq F(x)$ for all $x$ and $C^2:=\int F^2 dP<\infty$;
\item the collection $\text{Subgraph}(\mathcal{F})=\bigl\{\{(x,t)\in\mathcal{X}\times \mathbb{R}: f(x)\leq t\}: f\in \mathcal{F}\bigr\}$ is a Vapnik-Chervonenkis(VC) class of sets,
\end{enumerate}
then there exist positive constant $C_1,C_2$ for which
\begin{equation}
\label{eq:emp.lemma}
\mathbb{P}\left\{\sqrt{n}\sup_{f\in\mathcal{F}}\left|\frac{1}{n}\sum_{i\leq n}f(X_i)-\int f dP\right|>t\right\}\leq 4\mathbb{E}\exp\left(-\left(\frac{t}{C_2\|F\|_n}-1\right)^2\right)+4\mathbb{P}\Bigl\{2\|F\|_n> t/C_2\Bigr\},\;\;\; \forall t\geq C_1,
\end{equation}
where $\|F\|_n =n^{-1/2}\sqrt{\sum F(X_i)^2}$.
In particular, %if $\int F^2dP<\infty$, then~\eqref{eq:emp.lemma} implies the easy consequence that
$\sqrt{n}\sup_{f\in\mathcal{F}}\left|\frac{1}{n}\sum_{i\leq n}f(X_i)-\int f dP\right|=O_p(1)$.
\end{lemma}

 We note that more standard results could be used for concentration of measure over VC classes of Boolean functions, or over bounded classes of real functions. We use the lemma above because of its generality and to make our analysis self-contained. The proof of this result is included in Section~\ref{sec:pop}.

 Recall that $\text{Cov}_F(a,\{y>\tilde{q}_\tau\})=\overline{W}_{A-\mathbb{E}A}(\hat{\mu}_\tau,\hat{\nu}_\tau)$.
 We have
 \begin{align*}
 \sup_\tau\left|\text{Cov}_F\left(a,\{y>\tilde{q}_\tau(a,x)\}\right)\right|
 \leq
  \sup_{\mu,\nu\in\mathbb{R}}\left|\left(W_{A-\mathbb{E}A}(\mu,\nu)-\overline{W}_{A-\mathbb{E}A}(\mu,\nu)\right)\right|+ \sup_\tau\left|W_{A-\mathbb{E}A}\left(\hat{\mu}_\tau,\hat{\nu}_\tau\right)\right|.
\end{align*}
 We use Lemma~\ref{lmm:emp} to control the tail of  the first term
 using the VC class $\text{Subgraph}(\mathcal{F})$ where $f_{\mu,\nu}(a,r)=(a-\mathbb{E}a)\mathds{1}\{r>\mu a+\nu\}$ and $\mathcal{F}=\{f_{\mu,\nu}:\mu,\nu\in\mathbb{Q}\}$.
 For the second term we have
 \[
 W_{A-\mathbb{E}A}\left(\hat{\mu}_\tau,\hat{\nu}_\tau\right) = \frac{1}{n}\sum_{i\leq n}\left(A_i-\mathbb{E}A_i\right)\left\{R_i>\hat{\mu}_\tau A_i+\hat{\nu}_\tau\right\}.
 \]
 and we exploit the dual form of quantile regression in terms of rank scores
 together with large deviation bounds for sub-Gaussian random variables.

 The proof of Theorem~\ref{thm:risk} similarly exploits Lemma~\ref{lmm:emp}, but using
 the VC class $\text{Subgraph}(\mathcal{F})$ where $f_{\mu,\nu}(a,r)=\rho_\tau(r-\mu a-\nu)$ and $\mathcal{F}=\{f_{\mu,\nu}:\mu,\nu\in\mathbb{Q}\}$.

%% file: experiments.tex
% !TEX root = main.tex

\vskip20pt
\section{Experiments}
\label{sec:experiments}

\subsection{Experiments on synthetic data}
In this section we show experiments on synthetic data that verify our theoretical claims.
\footnote{Code and data for all experiments are available online at \url{https://drive.google.com/file/d/1Ibaq5VWaAE4539hec4-UdIOgPsNv0x_t/view?usp=sharing}}
The experiment is carried out in $N=10{,}000$ independent repeated trials. In each trial, $n=1{,}000$ data points $(X,A,Y)\in \mathbb{R}^p\times \{0,1\}\times \mathbb{R}$ are generated independently as follows:
\begin{itemize}
\item Let $p=20$. Generate $X$ from the multivariate distribution with correlated attributes: $X\sim \mathcal{N}(0,\Sigma)$, where the the covariance matrix $\Sigma\in \mathbb{R}^{p\times p}$ takes value $1$ for diagonal entries and $0.3$ for off-diagonal entries.
\item The protected attribute $A$ depends on $X$ through a logistic
model: $A\given X\sim \text{Bernoulli}(b)$ with
\begin{equation*}
b=\exp\left(X^T\gamma\right)/\left(1+\exp\left(X^T\gamma\right)\right).
\end{equation*}
\item Given $A,X$, generate $Y$ from a heteroscedastic model:
$Y\given A,X\sim \mathcal{N}\left(X^T\beta+\mu A, (X^T\eta)^2\right)$.
\end{itemize}
The parameters $\beta$, $\gamma$, $\eta$ are all generated independently from $\mathcal{N}(0,I_p)$ and stay fixed throughout all trials. The coefficient $\mu$ is set to be 3.

In each of the $N$ trials, conditional quantile estimators are trained on a training set of size $n/2$ and evaluated on the remaining size $n/2$ held out set. We train three sets of quantile  estimators at $\tau=0.5$:
\begin{enumerate}
\item Full quantile regression of $Y$ on $A$ and $X$.
\item Quantile regression of $Y$ on $X$ only.
\item Take the estimator from procedure 2 and correct it with the method described in Section~\ref{sec:results}.
\end{enumerate}

The average residuals $Y-\hat{q}_\tau(X,A)$ are then evaluated on the test set for the $A=0$ and $A=1$ subpopulations. In Figure~\ref{fig:synthetic} we display the histograms of these average residuals across all $N$ trials for the quantile regression estimator on $X$ (\ref{fig:residual.before}) and the corrected estimator (\ref{fig:residual.after}). In the simulation we are running, $A$ is positively correlated with the response $Y$. Therefore when $A$ is excluded from the regression, the quantile estimator underestimates when $A=1$ and overestimates when $A=0$. That is why we observe different residual distributions for the two subpopulations. This effect is removed once we apply the correction procedure, as shown in Figure~\ref{fig:fair.measure}.

%\begin{figure}
%\begin{center}
%\begin{tabular}{ccc}
%\includegraphics[width=.3\textwidth]{fig/hist_residual_wob} &
%\includegraphics[width=.3\textwidth]{fig/hist_residual_corrected} &
%\includegraphics[width=.3\textwidth]{fig/hist_fairness}
%\end{tabular}
%\end{center}
%\caption{\small \it Histograms of the average residuals of the conditional quantile estimators. The height of the bins are normalized so they integrate to one.
%Histograms and density estimates of the fairness measures the estimator obtained by running quantile regression on $X$, before (salmon) and after (light blue) the adjustment procedure. The histogram obtained from the full regression (black) serves as a benchmark for comparison. {\bf Separate captions for subfigures}}
%\label{fig:residual}
%\end{figure}

We also test whether our correction procedure corrects the unbalanced effective quantiles of an unfair initializer. In each trial we measure the fairness level of an estimator $\widehat{q}_\tau$ by the absolute difference $|\hat\tau_1 - \hat\tau_0|$ between the effective quantiles of the two subpopulations on a heldout set $S$, where $\hat{\tau}_a$ is defined as in~\eqref{eq:effective.quantile}.

We established in previous sections that quantile regression excluding attribute $A$ is in general not fair with respect to $A$. A histogram of the fairness measure obtained from this procedure is shown in Figure~\ref{fig:fair.measure} (salmon). Plotted together are the fairness measures after the correction procedure (light blue). For comparison we also include the histogram obtained from the full regression (black). Note that the full regression has the ``unfair" advantage of having access to all observations of $A$. Figure~\ref{fig:fair.measure} shows that the correction procedure pulls the fairness measure to a level comparable to that of a full regression, which as we argued in Section~\ref{sec:training}, produces $\sqrt{n}$-fair estimators.

\begin{figure}
    \centering
    \begin{subfigure}[b]{0.27\textwidth}
    \begin{centering}
        \includegraphics[width=\textwidth]{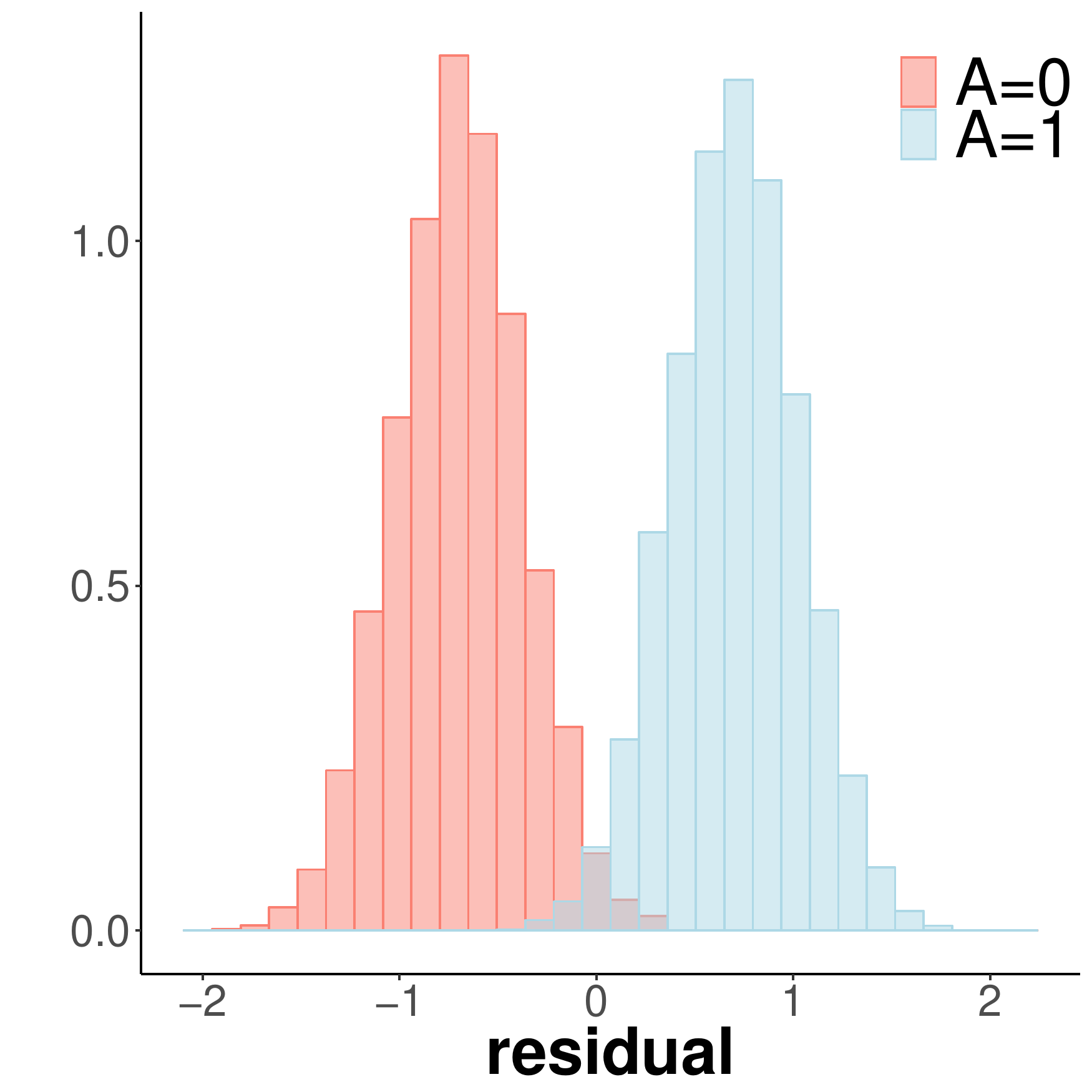}
        \caption[justification   = centering]{}
        \label{fig:residual.before}
     \end{centering}
    \end{subfigure}
    \begin{subfigure}[b]{0.27\textwidth}
    \begin{centering}
        \includegraphics[width=\textwidth]{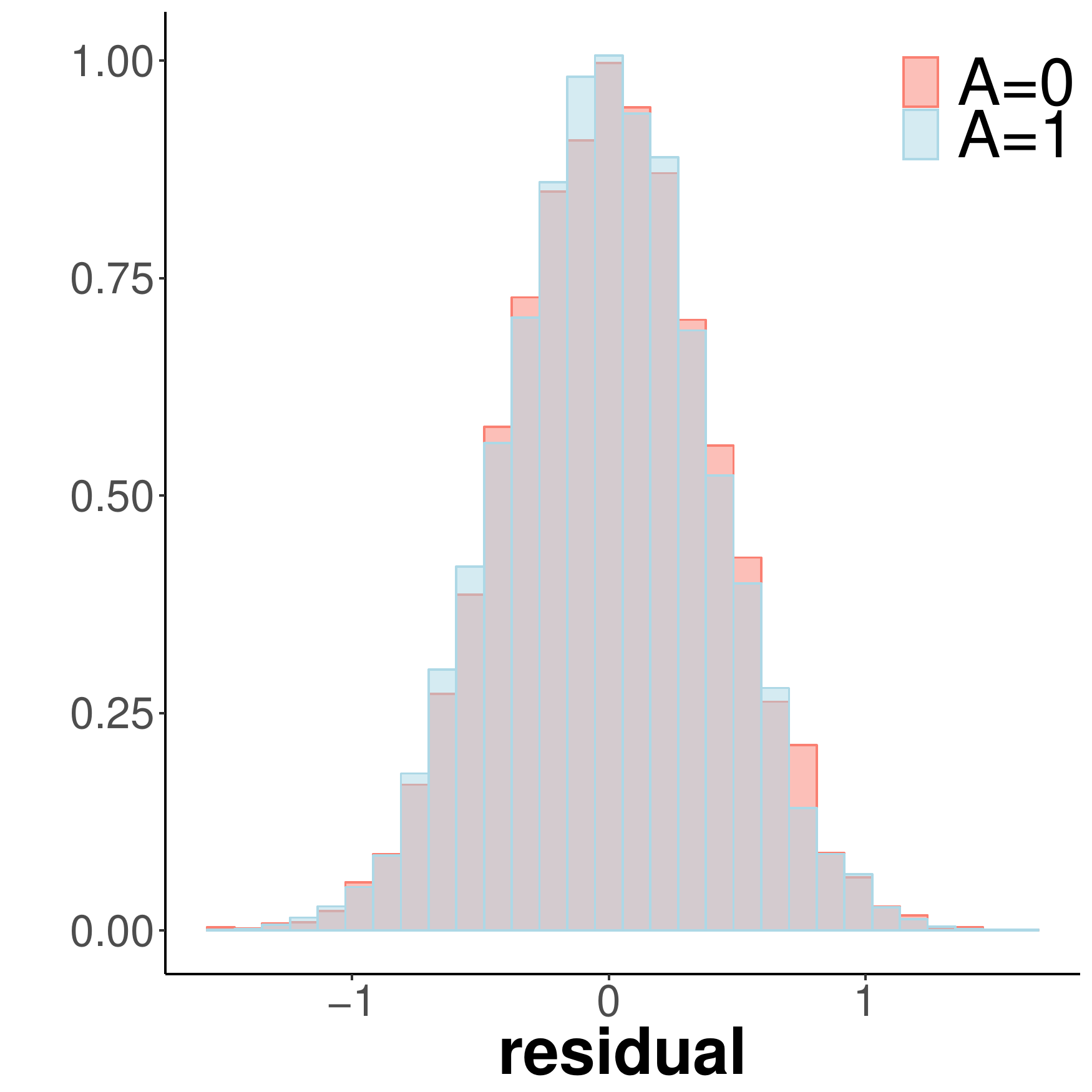}
        \caption[justification   = centering]{}
        \label{fig:residual.after}
    \end{centering}
    \end{subfigure}
    \hfill
    \begin{subfigure}[b]{0.30\textwidth}
    \begin{centering}
        \includegraphics[width=\textwidth]{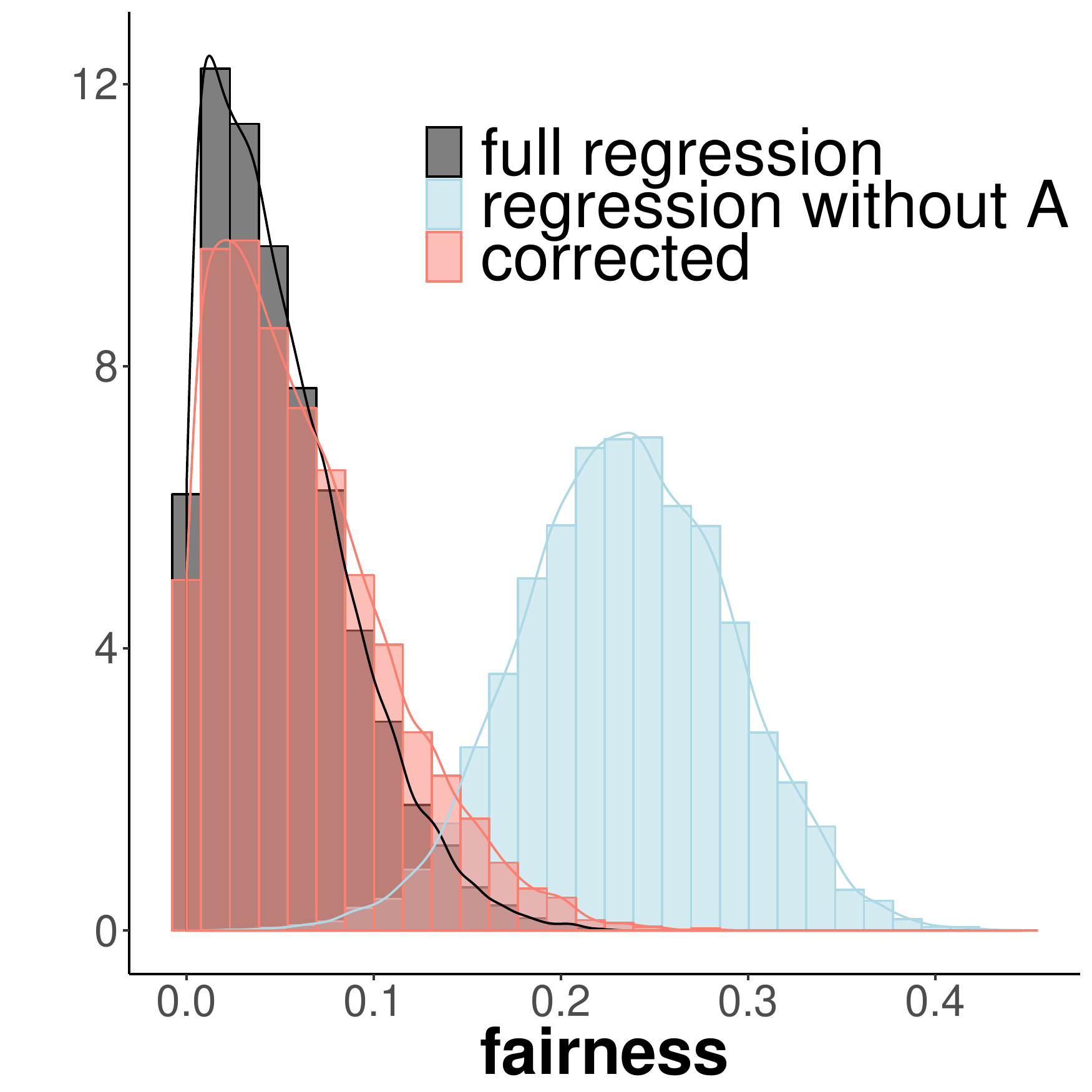}
        \caption[justification   = centering]{}
        \label{fig:fair.measure}
    \end{centering}
    \end{subfigure}
    \caption{\small \it From left to right: (a): histograms of average residuals for quantile regression of $Y$ on $X$ only; (b): histograms of average residuals for the corrected quantile estimators; (c): histograms and density estimates of the fairness measures obtained by running quantile regression on $X$, before (salmon) and after (light blue) the adjustment procedure. The histogram from the full regression (black) serves as a benchmark for comparison.}
    \label{fig:synthetic}
\end{figure}

%% file: birth.tex
% !TEX root = main.tex

\subsection{Birthweight data analysis}
\label{sec:birth}

%\begin{figure*}[ht]
%\begin{center}
%\def\pwidth{.15}
%\begin{tabular}{cccc}
%\includegraphics[width=\pwidth\textwidth]{fig/qrcoeff-w-wo-black-1} &
%\includegraphics[width=\pwidth\textwidth]{fig/qrcoeff-w-wo-black-2} &
%\includegraphics[width=\pwidth\textwidth]{fig/qrcoeff-w-wo-black-3} &
%\includegraphics[width=\pwidth\textwidth]{fig/qrcoeff-w-wo-black-16} \\
%\includegraphics[width=\pwidth\textwidth]{fig/qrcoeff-w-wo-black-4} &
%\includegraphics[width=\pwidth\textwidth]{fig/qrcoeff-w-wo-black-5} &
%\includegraphics[width=\pwidth\textwidth]{fig/qrcoeff-w-wo-black-6} &
%\includegraphics[width=\pwidth\textwidth]{fig/qrcoeff-w-wo-black-7} \\
%\includegraphics[width=\pwidth\textwidth]{fig/qrcoeff-w-wo-black-8} &
%\includegraphics[width=\pwidth\textwidth]{fig/qrcoeff-w-wo-black-9} &
%\includegraphics[width=\pwidth\textwidth]{fig/qrcoeff-w-wo-black-10} &
%\includegraphics[width=\pwidth\textwidth]{fig/qrcoeff-w-wo-black-11} \\
%\includegraphics[width=\pwidth\textwidth]{fig/qrcoeff-w-wo-black-12} &
%\includegraphics[width=\pwidth\textwidth]{fig/qrcoeff-w-wo-black-13} &
%\includegraphics[width=\pwidth\textwidth]{fig/qrcoeff-w-wo-black-14} &
%\includegraphics[width=\pwidth\textwidth]{fig/qrcoeff-w-wo-black-15}
%\end{tabular}
%\end{center}
%\caption{Quantile regression coefficients, including the protected
%  race variable $A$ (solid, salmon), and excluding $A$ (long dashed,
%  light blue). When the race variable $A$ is excluded, the variable ``married'' can be
% seen as serving as a kind of proxy.
% }
%\label{fig:all}
%\end{figure*}

\afterpage{
\begin{figure}[t!]
\begin{center}
\def\pwidth{.16}
\begin{tabular}{ccccc}
\includegraphics[width=\pwidth\textwidth]{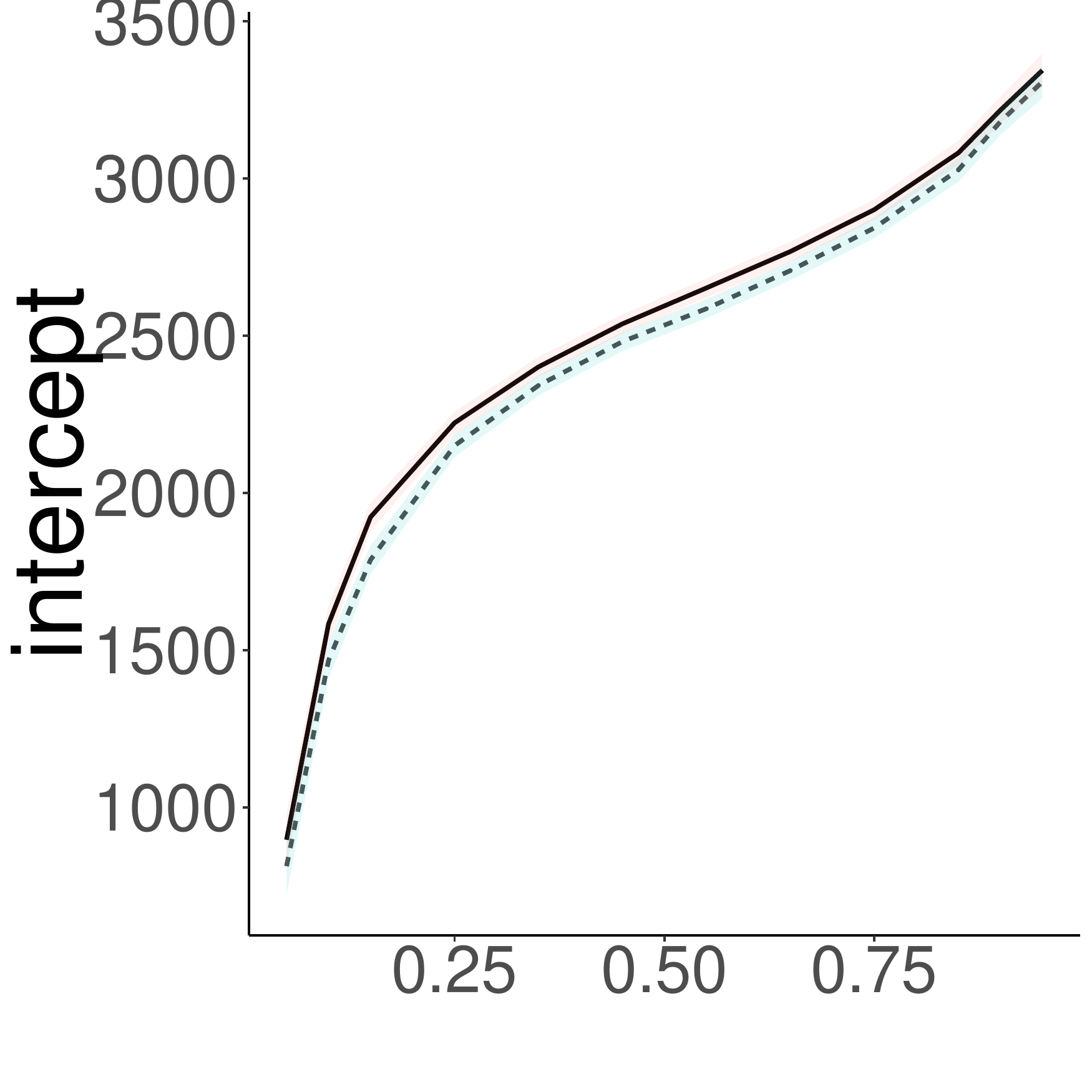} &
\includegraphics[width=\pwidth\textwidth]{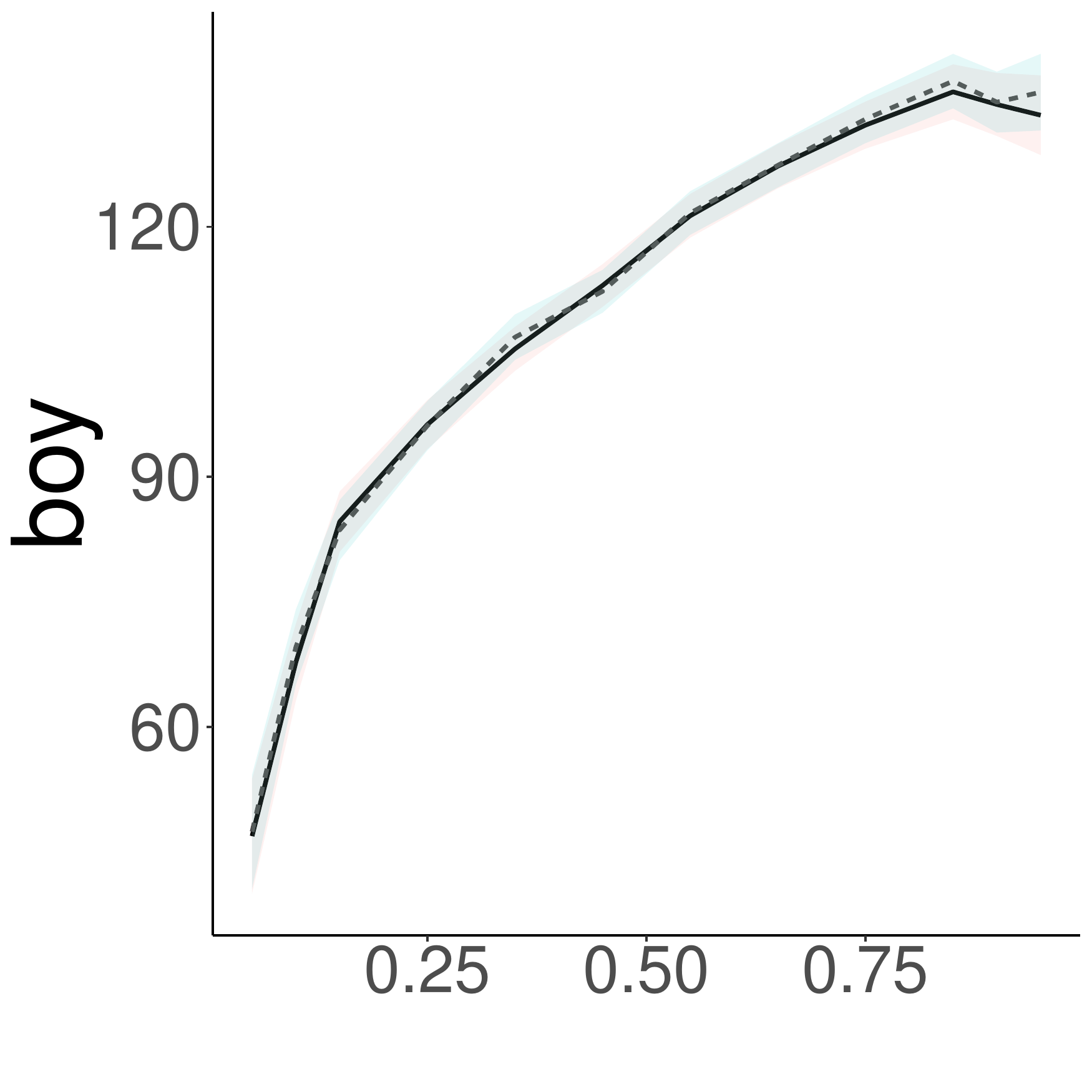} &
\includegraphics[width=\pwidth\textwidth]{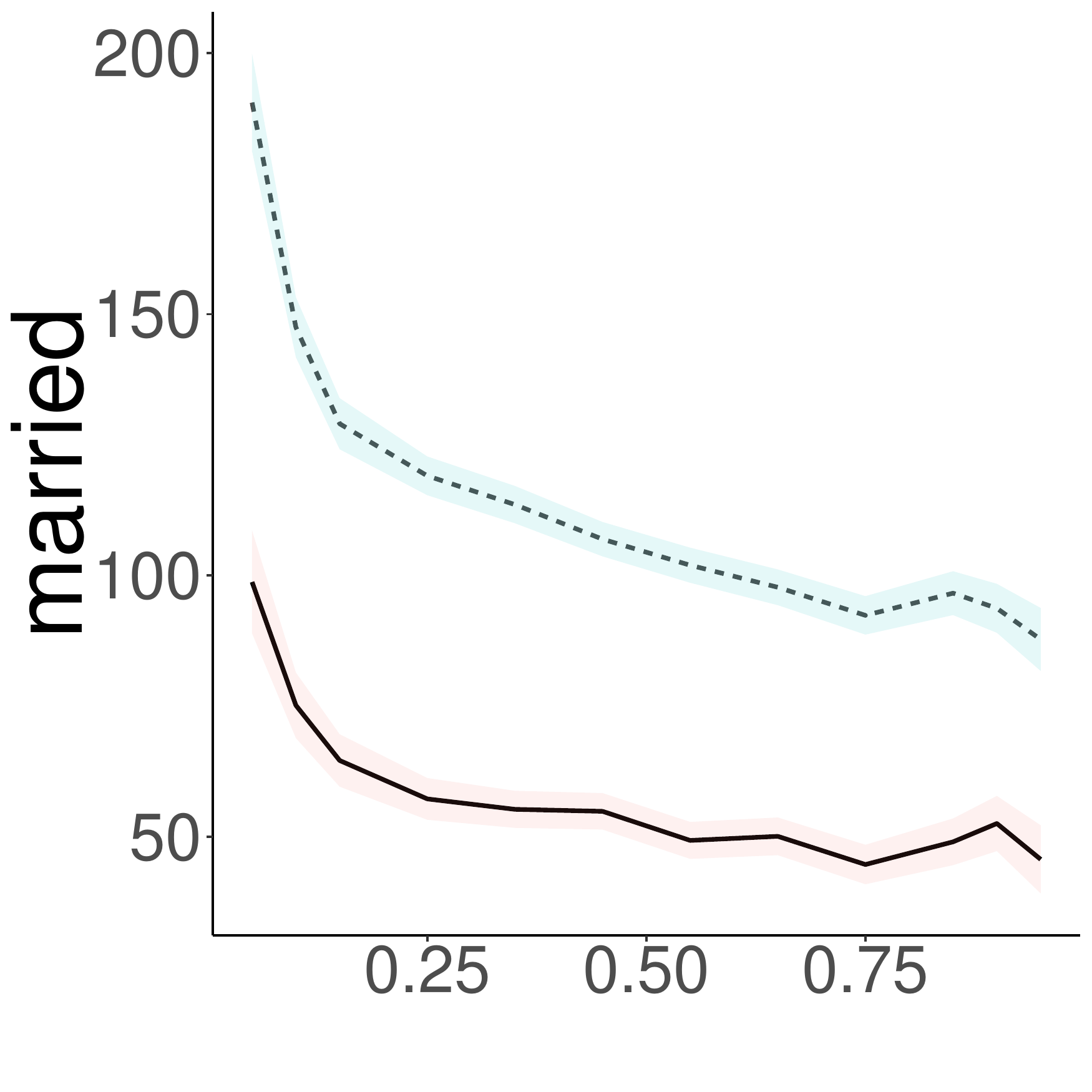} &
\includegraphics[width=\pwidth\textwidth]{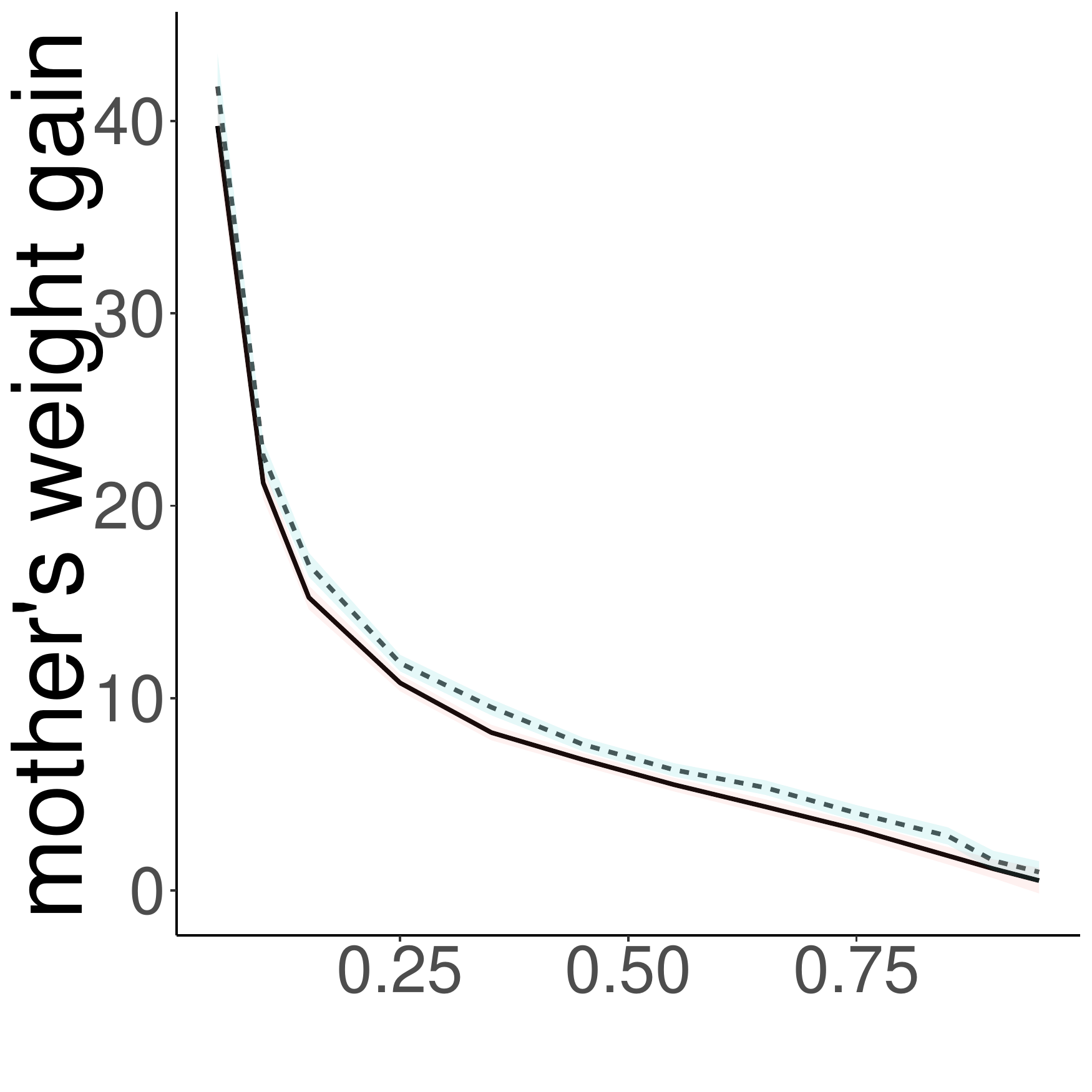} &
\includegraphics[width=\pwidth\textwidth]{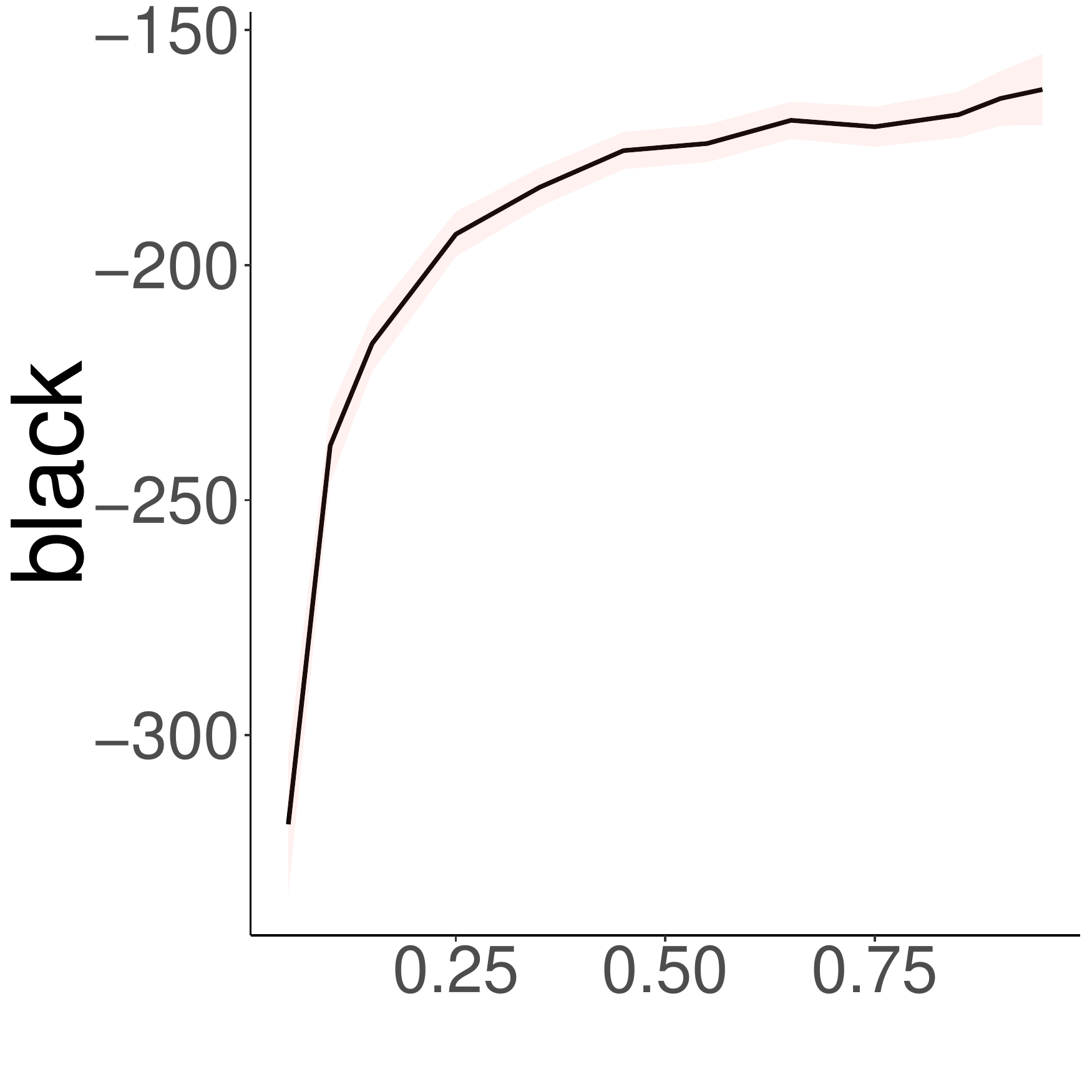} \\
\includegraphics[width=\pwidth\textwidth]{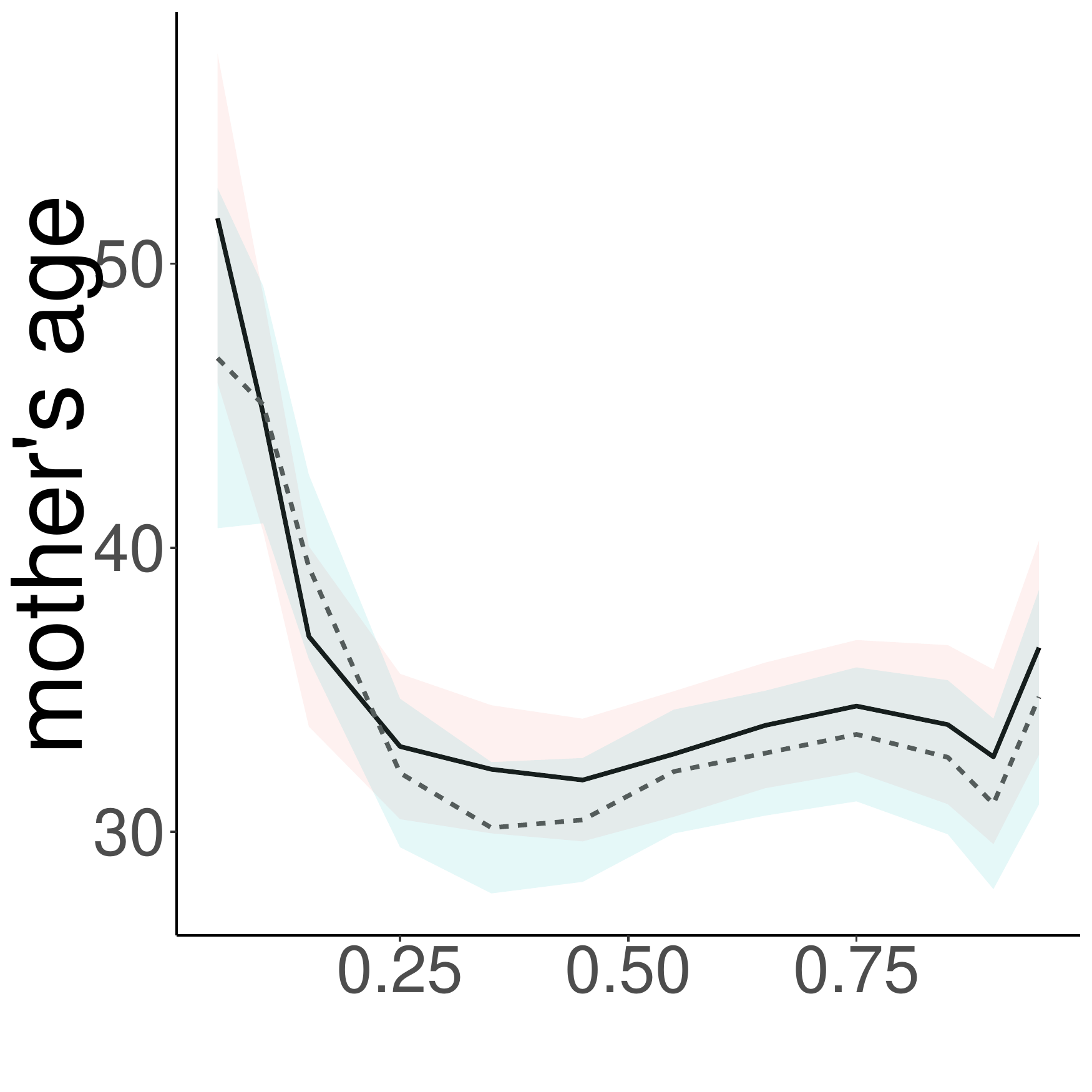} &
\includegraphics[width=\pwidth\textwidth]{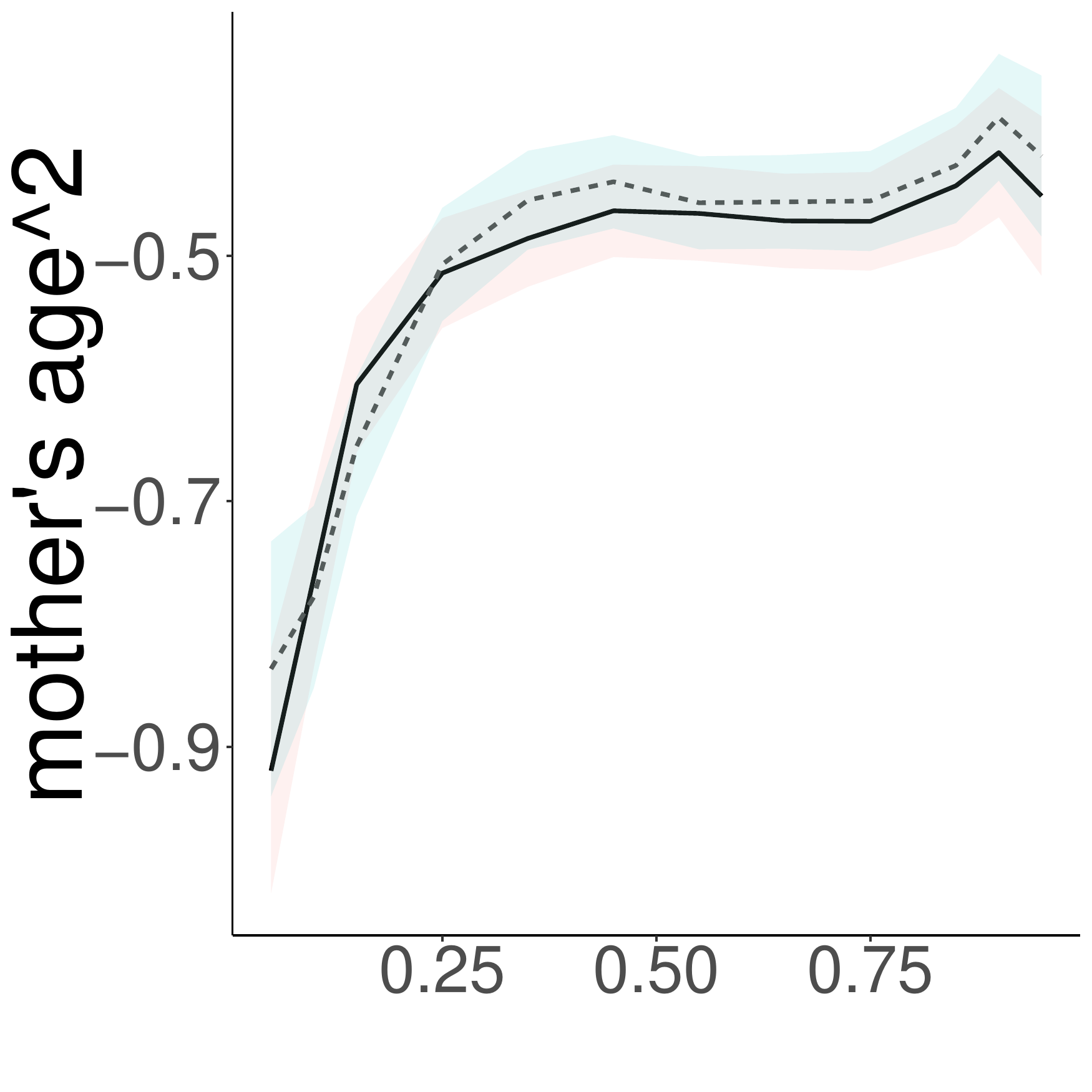} &
\includegraphics[width=\pwidth\textwidth]{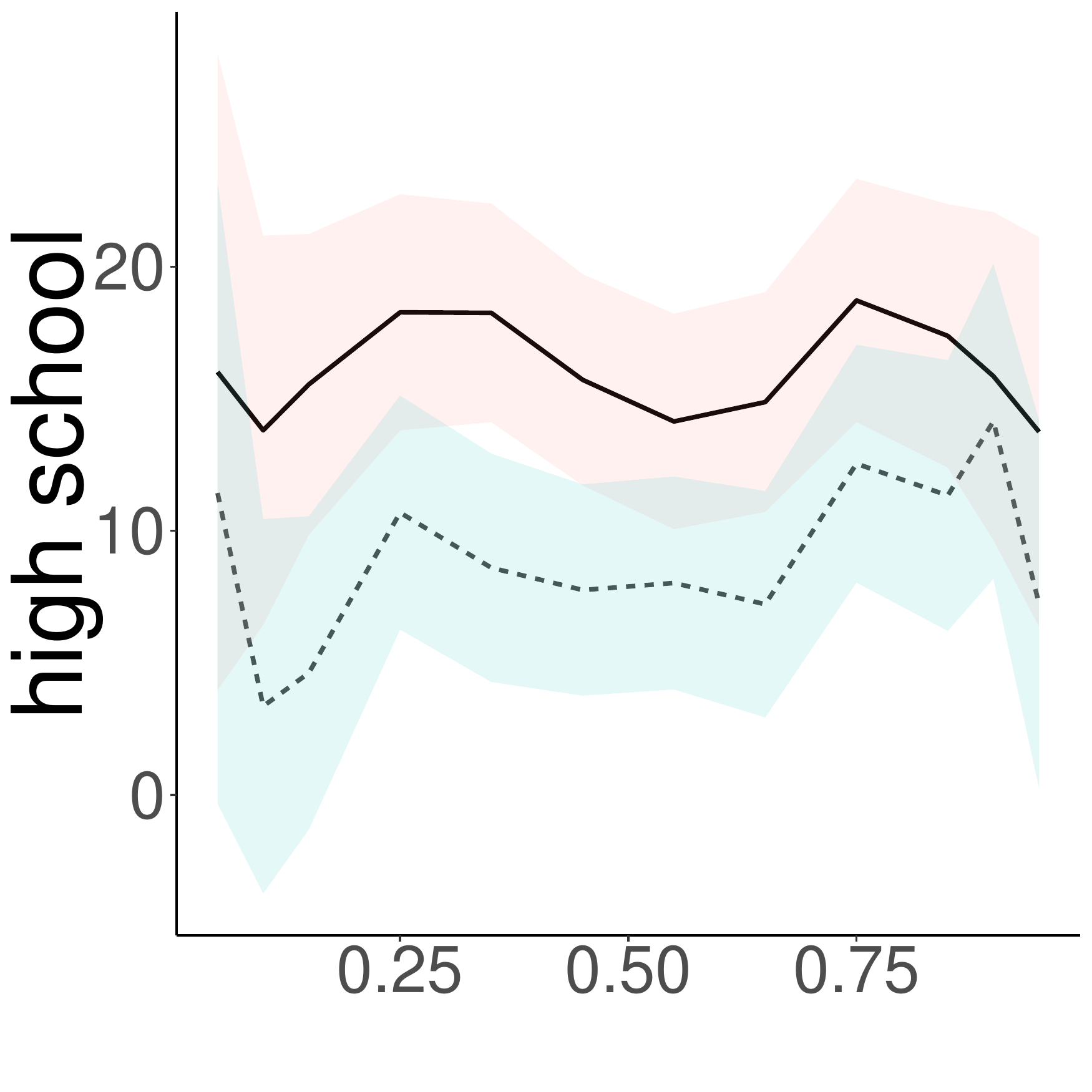} &
\includegraphics[width=\pwidth\textwidth]{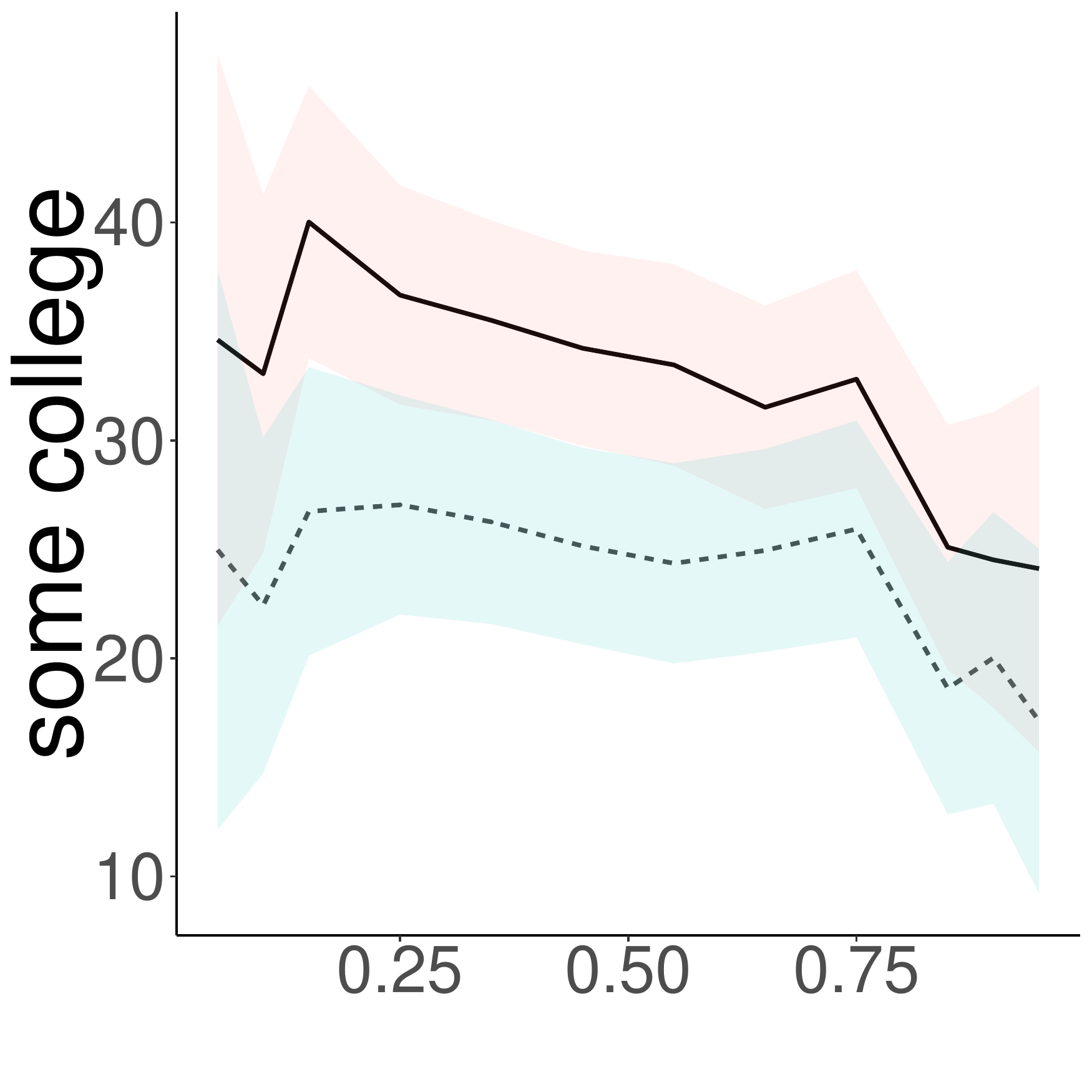} &
\includegraphics[width=\pwidth\textwidth]{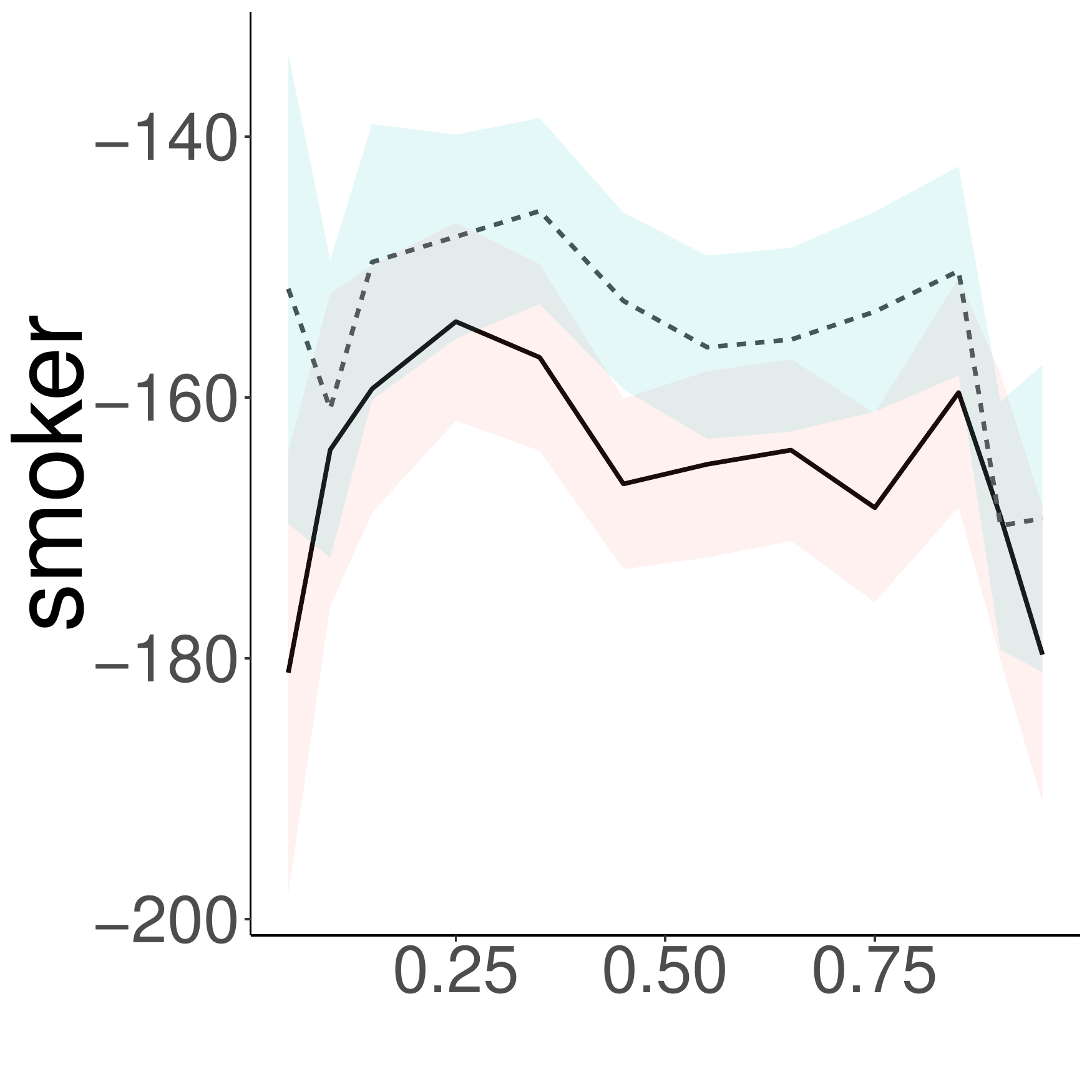} \\
\end{tabular}
\end{center}
\caption{\small\it Quantile regression coefficients for birth data. The quantile $\tau$ runs along horizontal axis; curves are the coefficients $\hat \beta_\tau$; unit is grams. Solid/salmon: race is included in the model; dashed/blue: race excluded. When race is excluded, the strongly correlated variable ``married'' can be seen as serving as a kind of proxy.
 }
 \vskip-10pt
\label{fig:all}
\end{figure}
\begin{figure}[h!]
    \centering
    \begin{subfigure}[t]{0.35\textwidth}
        \centering
\begin{footnotesize}
  \begin{tabular}{| c || c | c | c | c | c |}
    \hline\hline
    target & $5$ & $25$ & $50$ & $75$  \\ \hline\hline
    $\hat{\tau}_0$(before) & $4.42$ & $23.14$ & $47.87$ & $73.12$ \\
    $\hat{\tau}_1$(before) & $7.91$ & $33.80$ & $60.02$ & $82.46$ \\
    $\hat{\tau}_0$(after) & $5.03$ & $25.01$ & $49.77$ & $74.61$ \\
    $\hat{\tau}_1$(after) & $5.02$ & $24.02$ & $49.95$ & $74.62$ \\
    \hline
  \end{tabular}
\end{footnotesize}
        \caption{Quantiles before and after adjustment.}
    \end{subfigure}%
    ~
    \begin{subfigure}[t]{0.7\textwidth}
        \centering
        \begin{tabular}{cc}
        \quad\quad\includegraphics[width=.35\textwidth]{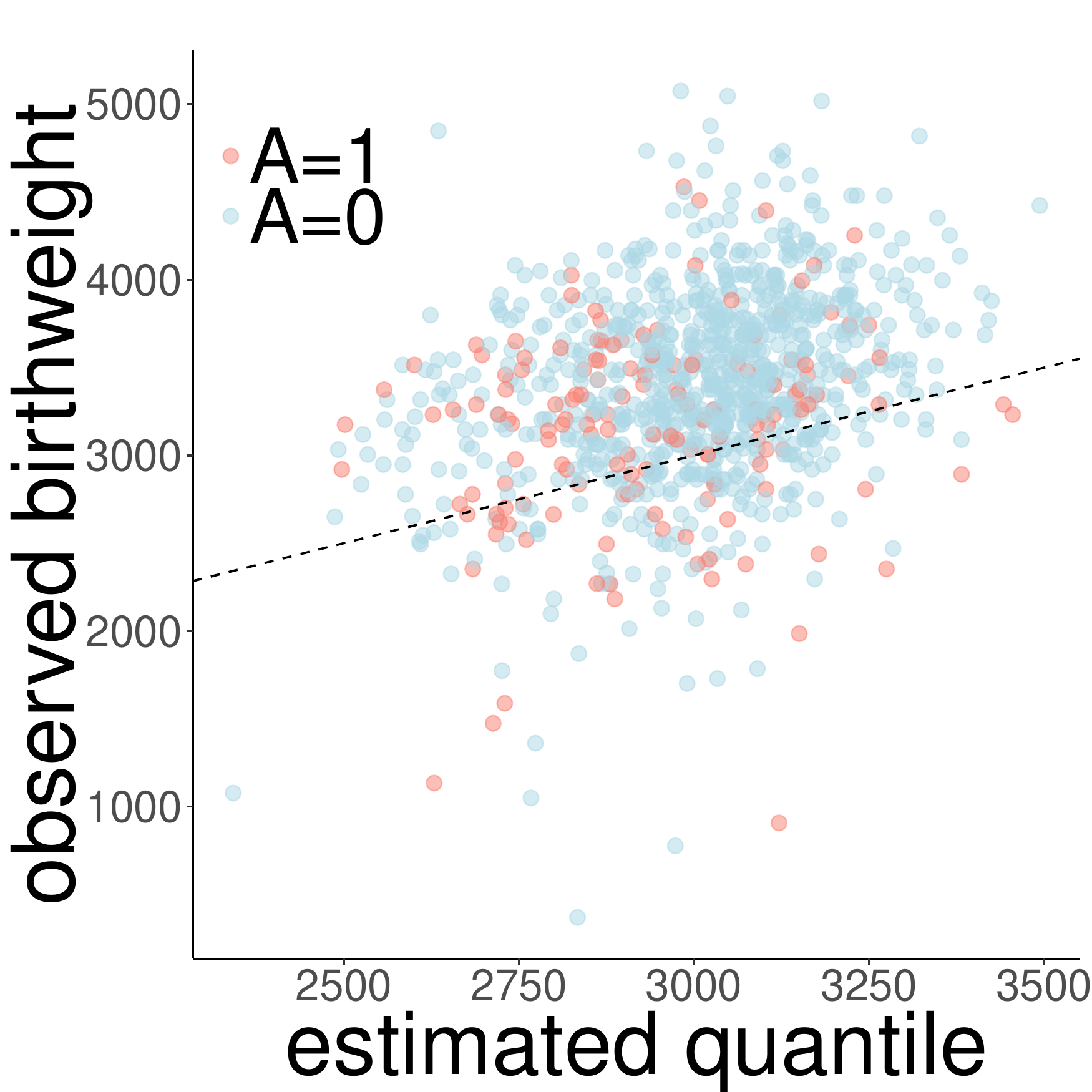} &
        \includegraphics[width=.35\textwidth]{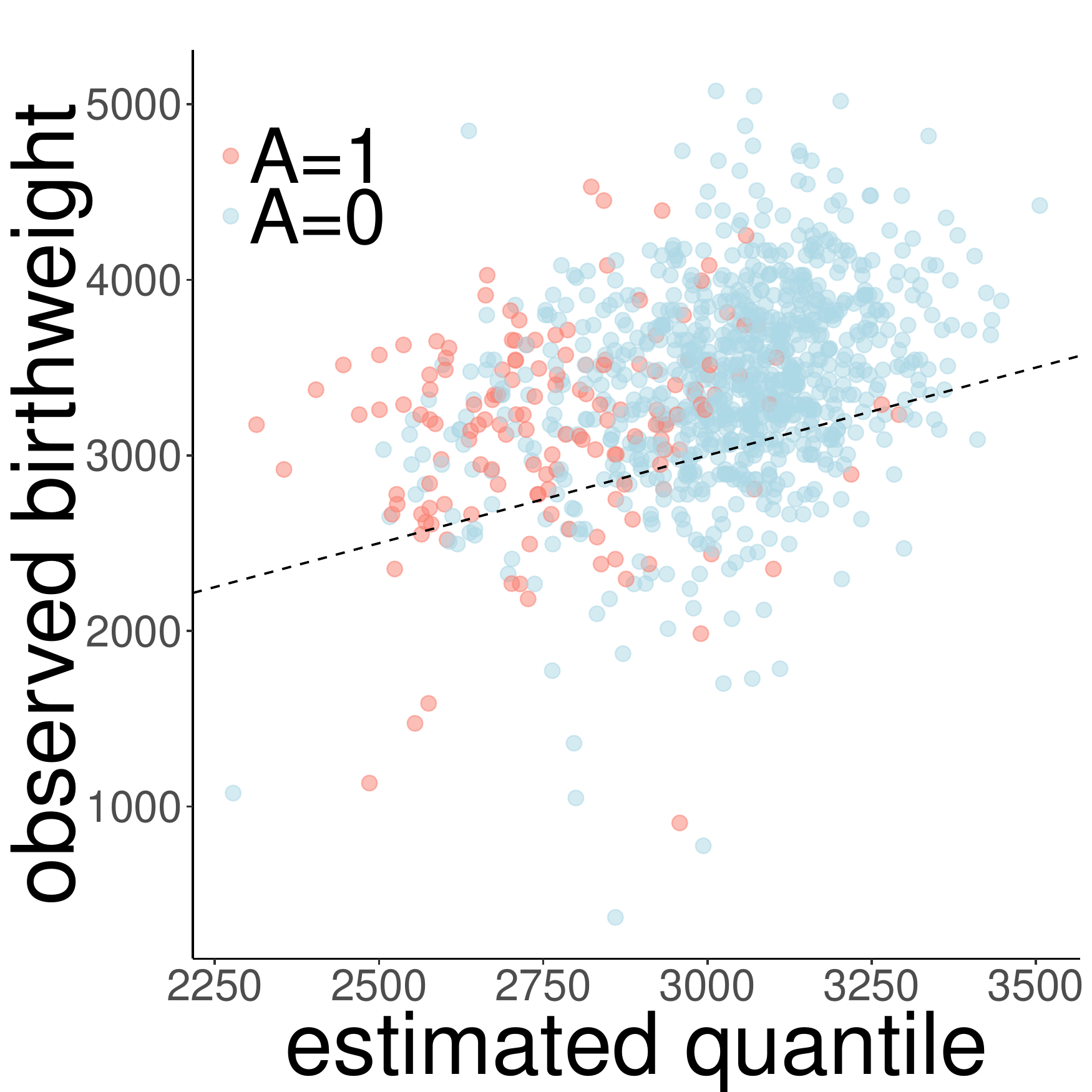}
        \end{tabular}
        \caption{Birth weights before and after adjustment.}
    \end{subfigure}
    \caption{\small \it Left: the effect of the adjustment procedure. Right: scatter plots of observed birth weights against estimated $20\%$ conditional quantiles over the test set, before and after adjustment.}
\label{fig:effective}
\end{figure}
}

The birth weight dataset from
\cite{abrevaya:01}, which is analyzed by
\cite{koenker2001quantile}, includes the weights
of 198,377 newborn babies, and other attributes of the babies
and their mothers, such as the baby's gender,
whether or not the mother is married, and the mother's age. One of the
attributes includes information about the race of the mother, which we treat as the
protected attribute $A$. The variable $A$ is binary---black ($A=1$) or
not black ($A=0$). The birth weight is reported in grams. The other attributes include education of the mother,
prenatal medical care, an indicator of whether the mother smoked during pregnancy,
and the mother's reported weight gain during pregnancy.

Figure~\ref{fig:all} shows the coefficients $\widehat{\beta}_\tau$ obtained
by fitting a linear quantile regression model, regressing birth weight
on all other attributes.  The model is fit two ways,
either including the protected
race variable $A$ (solid, salmon confidence bands), or excluding $A$ (long dashed,
light blue confidence bands).
The top-right figure shows that babies of black mothers weigh less on
average, especially near the lower quantiles where they weigh nearly
300 grams less compared to babies of nonblack mothers.  A description
of other aspects of this linear model is given by \cite{koenker2001quantile}.
A  striking aspect of the plots is the disparity between birth weights of infants born to black and nonblack
mothers, especially at the left tail of the
distribution. In particular, at the 5th percentile of the conditional distribution,
the difference is more than 300 grams. Just as striking is the
observation that when the race attribute $A$ is excluded from the model, the variable
``married,'' with which it has a strong negative correlation, effectively serves
as a proxy, as seen by the upward shift in its regression
coefficients. However, this and the other variables do not completely
account for race, and as a result the model overestimates the weights
of infants born to black mothers, particularly at the lower
quantiles.

To correct for the unfairness of $\widehat{q}_\tau$, we apply the correction procedure described in Section~\ref{sec:results}. For the target quantile $\tau=20\%$, the corrected estimator $\tilde{q}_\tau$ achieves effective quantiles $20.4\%$ for the black population and $20.1\%$ for the nonblack population. Table~\ref{fig:effective} (left) shows the effective quantiles at a variety of quantile levels. We see that the correction procedure consistently pulls the effective quantiles for both subpopulations closer to the target quantiles.

For 1000 randomly selected individuals from the test set, Figure~\ref{fig:effective} (right) shows their observed birth weights plotted against the conditional quantile estimation at $\tau=20\%$ before (left) and after (right) the correction. The dashed line is the identity. When $A$ is not included in the quantile regression, the conditional quantiles for the black subpopulation are overestimated. Our procedure achieves fairness correction by shifting the estimates for the $A_i=1$ data points smaller (to the left) and shifting the $A_i=0$ data points larger (to the right). After the correction, the proportion of data points that satisfy $Y\leq\tilde{q}_\tau$ are close to the target $20\%$ for both subpopulations.

%% file: discussion.tex
% !TEX root = main.tex

\section{Discussion}

In this paper we have studied the effects of excluding a distinguished
attribute from quantile regression estimates, together with procedures
to adjust for the bias in these estimates through post-processing.
The linear programming basis for quantile regression
leads to properties and analyses that complement what has
appeared previously in the fairness literature.
Several extensions of the work presented here
should be addressed in future work. For example, the generality of
the concentration result of Lemma~\ref{lmm:emp} could allow the extension
of our results to multiple attributes of different types. In the fairness
analysis in Section~\ref{sec:proof} we used a linear quantile regression in the adjustment step, which
allows us to more easily leverage previous statistical analyses
\cite{gutenbrunner1992regression} on quantile rank scores. Nonparametric
methods would be another interesting direction to explore.

The birth data studied here has been instrumental in developing our
thinking on fairness for quantile regression. It will be interesting
to investigate the ideas introduced here for other data sets. If the tail
behaviors, including outliers, of the conditional distributions for a
set of subpopulations are very different, and the identification of
those subpopulations is subject to privacy restrictions or other
constraints that do not reveal them in the data, the issue of bias in
estimation and decision making will come into play.

%% file: proof.tex
% !TEX root = main.tex

\vskip20pt
\section{Proofs}\label{sec:pop}

\begin{proof}[Proof of Lemma~\ref{lmm:emp}]
To prove the lemma we first transform the problem into bounding the tail of a Rademacher process via a symmetrization technique. Let $\epsilon_i$ be distributed {\it i.i.d.} Rademacher ($\mathbb{P}\{\epsilon_i=1\}=\mathbb{P}\{\epsilon_i=-1\}=1/2$). Write $P_n$ for the empirical (probability) measure that puts mass $n^{-1}$ at each $X_i$. We claim that for all $t>2\sqrt{2}C=:C_1$,
\begin{equation}
\label{eq:symm}
\mathbb{P}\left\{\sqrt{n}\sup_{f\in\mathcal{F}}\left|\int f dP_n-\int f dP\right|>t\right\}\leq
4\mathbb{P}\left\{\sqrt{n}\sup_{f\in \mathcal{F}}\left|\frac{1}{n}\sum_{i\leq n}\epsilon_i f(X_i)\right|>\frac{t}{4}\right\}.
\end{equation}
{\bf Proof of~\eqref{eq:symm}}: Let $\tilde{X}_1,...,\tilde{X}_n$ be independent copies of $X_1,..., X_n$ and let $\tilde{P}_n$ be the corresponding empirical measure. Define events
\[
\mathcal{A}_f = \left\{\sqrt{n}\left|\int f dP_n- \int f dP\right|>t\right\};\;\;\;\text{and }\mathcal{B}_f=\left\{\sqrt{n}\left|\int f d\tilde{P}_n-\int f dP\right|\leq \frac{t}{2}\right\}.
\]
For all $t>C_1$,
\[
\mathbb{P}\mathcal{B}_f= 1-\mathbb{P}\left\{\sqrt{n}\left|\int f d\tilde{P}_n-\int f dP\right|>\frac{t}{2}\right\}
\geq 1-\frac{\text{Var} f(X_1)}{(t/2)^2}
\geq 1-\frac{\int F^2 dP}{(t/2)^2}
\geq \frac{1}{2}.
\]
On the other hand, because $\mathcal{F}$ is countable, we can always find mutually exclusive events $\mathcal{D}_f$ for which
\[
\mathbb{P}\cup _{f\in\mathcal{F}}\mathcal{A}_f=\mathbb{P}\cup _{f\in\mathcal{F}}\mathcal{D}_f = \sum_{f\in\mathcal{F}}\mathbb{P}\mathcal{D}_f.
\]
Since $2\mathbb{P}\mathcal{B}_f\geq 1$ for all $f$, the above is upper bounded by $2\sum_{f\in\mathcal{F}}\mathbb{P}\mathcal{D}_f\mathbb{P} \mathcal{B}_f$. From independence of $X$ and $\tilde{X}$, it can be rewritten as
\[
2\sum_{f\in\mathcal{F}}\mathbb{P}(\mathcal{D}_f\cap \mathcal{B}_f)=2\mathbb{P} \cup_{f\in\mathcal{F}}(\mathcal{D}_f\cap \mathcal{B}_f)\leq 2\mathbb{P}\cup_{f\in\mathcal{F}}(\mathcal{A}_f \cap \mathcal{B}_f),
\]
which is no greater than 2$\mathbb{P}\{\sqrt{n}\sup_f |\int f dP_n-\int f d\tilde{P}_n|>t/2\}$ since
\[
\mathcal{A}_f\cap \mathcal{B}_f\subset \left\{\sqrt{n}\left|\int f dP_n-\int f d\tilde{P}_n\right|>t/2\right\}.
\]

Because $\tilde{X}_i$ is an independent copy of $X_i$, by symmetry $f(X_i)-f(\tilde{X}_i)$ and $\epsilon_i (f(X_i)-f(\tilde{X}_i))$ are equal in distribution. Therefore
\begin{align*}
& \mathbb{P}\left\{\sqrt{n}\sup_{f\in\mathcal{F}}\left|\int f dP_n-\int f dP\right|>t\right\}\\
= & \mathbb{P}\cup _{f\in\mathcal{F}}\mathcal{A}_f \leq 2\mathbb{P}\left\{\sqrt{n}\sup_{f\in\mathcal{F}}\left|\frac{1}{n}\sum_{i\leq n}\epsilon_i (f(X_i)-f(\tilde{X}_i))\right|>\frac{t}{2}\right\}\\
\leq & 2\mathbb{P}\left\{\sqrt{n}\sup_{f\in\mathcal{F}}\left|\frac{1}{n}\sum_{i\leq n}\epsilon_if(X_i)\right|>\frac{t}{4}\right\} + 2\mathbb{P}\left\{\sqrt{n}\sup_{f\in\mathcal{F}}\left|\frac{1}{n}\sum_{i\leq n}\epsilon_if(\tilde{X}_i)\right|>\frac{t}{4}\right\}\\
= & 4\mathbb{P}\left\{\sqrt{n}\sup_{f\in\mathcal{F}}\left|\frac{1}{n}\sum_{i\leq n}\epsilon_if(X_i)\right|>\frac{t}{4}\right\}.
\end{align*}
That concludes the proof of~\eqref{eq:symm}.

Denote as $Z_n(f)$ the Rademacher process $n^{-1/2}\sum \epsilon_i f(X_i)$. Let $\mathbb{P}_X$ be the probability measure of $\epsilon$ conditioning on $X$. By independence of $\epsilon$ and $X$, $\epsilon_i$ is still Rademacher under $\mathbb{P}_X$, and it is sub-Gaussian with parameter 1. This implies that for all $f,g\in\mathcal{F}$, $Z_n(f)-Z_n(g)$ is sub-Gaussian with parameter$\sqrt{\int (f-g)^2 dP_n}$ under $\mathbb{P}_X$. In other words,
\[
\mathbb{P}_X\left\{\left| Z_n(f)-Z_n(g)\right|>2\sqrt{\int (f-g)^2 dP_n}\sqrt{u}\right\}\leq 2e^{-u},\;\;\;\forall u>0.
\]
We have shown that conditioning on $X$, $Z_n(f)$ is a process with sub-Gaussian increments controlled by the $\mathcal{L}^2$ norm with respect to $P_n$. For brevity write $\|f\|$ for $\sqrt{\int f^2 dP_n}$. Apply Theorem 3.5 in~\cite{dirksen2015tail} to deduce that there exists positive constant $C_3$, such that for all $f_0\in\mathcal{F}$,
\begin{equation}
\label{eq:chaining}
\mathbb{P}_X\left\{\sup_{f\in\mathcal{F}}\left|Z_n(f)-Z_n(f_0)\right|\geq C_3\left(\Delta(\mathcal{F},\|\cdot\|)\sqrt{u}+\gamma_2(\mathcal{F},\|\cdot\|)\right)\right\}\leq e^{-u}\;\;\;\forall u\geq 1,
\end{equation}
where $\Delta(\mathcal{F},\|\cdot\|)$ is the diameter of $\mathcal{F}$ under the metric $\|\cdot\|$, and $\gamma_2$ is the generic chaining functional that satisfies
\[
\gamma_2(\mathcal{F},\|\cdot\|)\leq C_4 \int_0^{\Delta(\mathcal{F},\|\cdot\|)} \sqrt{\log N(\mathcal{F},\|\cdot\|,\delta)}d\delta
\]
for some constant $C_4$. Here $N(\mathcal{F},\|\cdot\|,\delta)$ stands for the $\delta$-covering number of $\mathcal{F}$ under the metric $\|\cdot\|$. We should comment that the generic chaining technique by~\cite{dirksen2015tail} is a vast overkill for our purpose. With some effort the large deviation bounds we need can be derived using the classical chaining technique.

Because $|f|\leq F$ for all $f\in\mathcal{F}$, we have $\Delta(\mathcal{F},\|\cdot\|)\leq 2\|F\|$, so that
\begin{align}
\nonumber & \int_0^{\Delta(\mathcal{F},\|\cdot\|)} \sqrt{\log N(\mathcal{F},\|\cdot\|,\delta)}d\delta\\
\label{eq:c.o.v.} \leq & \int_0^{2\|F\|} \sqrt{\log N(\mathcal{F},\|\cdot\|,\delta)}d\delta
= 2\|F\|\int_0^1 \sqrt{\log N(\mathcal{F},\|\cdot\|,2\delta\|F\|)}d\delta
\end{align}
via change of variables. To bound the covering number, invoke the assumption that $\text{Subgraph}(\mathcal{F})$ is a VC class of sets. Suppose the VC dimension of $\text{Subgraph}(\mathcal{F})$ is $V$. By Lemma 19 in~\cite{nolan1987u}, there exists positive constant $C_5$ for which the $\mathcal{L}^1(Q)$ covering numbers satisfy
\[
N\left(\mathcal{F},\mathcal{L}^1(Q),\delta \int F dQ\right)\leq (C_5/\delta)^V
\]
for all $0<\delta\leq 1$ and any $Q$ that is a finite measure with finite support on $\mathcal{X}$. Choose $Q$ by $dQ/dP_n=F$. Choose $f_1,..., f_N\in\mathcal{F}$ with $N=N(\mathcal{F},\mathcal{L}^1(Q),\delta \int F dQ)$ and $\min_i \int |f-f_i| dQ \leq \delta \int F dQ$ for each $f\in\mathcal{F}$. Suppose $f_i$ achieves the minimum. Since $F$ is an envelope function for both $f$ and $f_i$,
\[
\int \left|f-f_i\right|^2 dP_n \leq \int 2F \left|f-f_i\right| dP_n,
\]
which by definition of $Q$, is equal to
\[
2\int \left|f-f_i\right| dQ\leq 2\delta \int F dQ = 2\delta\int F^2 dP_n.
\]
Take square roots on both sides to deduce that
\[
N\left(\mathcal{F},\|\cdot\|, 2\delta \|F\|\right)\leq (C_5/\delta^2)^V.
\]
Plug into~\eqref{eq:c.o.v.} this upper bound on the covering number to deduce that the integral in~\eqref{eq:c.o.v.} converges, and $\gamma_2(\mathcal{F},\|\cdot\|)$ is no greater than a constant multiple of $\|F\|$. Recall that we also have $\Delta(\mathcal{F},\|\cdot\|)\leq 2\|F\|$. From~\eqref{eq:chaining}, there exists positive constant $C_6$ for which
\[
\mathbb{P}_X\left\{\sup_{f\in\mathcal{F}}\left|Z_n(f)-Z_n(f_0)\right|\geq C_6 \|F\|\left(\sqrt{u}+1\right)\right\}\leq e^{-u}\;\;\;\forall u\geq 1.
\]
Take $f_0=0$ so we have $Z_n(f_0)=0$. If the zero function does not belong in $\mathcal{F}$, including it in $\mathcal{F}$ does not disrupt the VC set property, and all previous analysis remains valid for $\mathcal{F}\cup \{0\}$. Letting $u=(t/4C_6 \|F\|-1)^2$ yields
\[
\mathbb{P}_X \left\{\sup_{f\in\mathcal{F}}|Z_n(f)|>\frac{t}{4}\right\}\leq \exp\left(-\left(\frac{t}{4C_6\|F\|}-1\right)^2\right),\;\;\; \forall t\geq 8C_6 \|F\|.
\]
Under $\mathbb{P}$, $\|F\|$ is no longer deterministic. Divide the probability space according to the event $\{t\geq 8C_6 \|F\|\}$:
\begin{align*}
\mathbb{P}\left\{\sup_{f\in\mathcal{F}}|Z_n(f)|>\frac{t}{4}\right\}
\leq &\mathbb{E}\mathds{1}\{t\geq 8C_6 \|F\|\}\mathbb{P}_X \left\{\sup_{f\in\mathcal{F}}|Z_n(f)|>\frac{t}{4}\right\} + \mathbb{P}\{t< 8C_6 \|F\|\}\\
\leq &\mathbb{E}\mathds{1}\{t\geq 8C_6 \|F\|\}\exp\left(-\left(\frac{t}{4C_6\|F\|}-1\right)^2\right)+
\mathbb{P}\{t< 8C_6 \|F\|\}.
\end{align*}
Choose $C_2=4C_6$ and~\eqref{eq:emp.lemma} follows.

\end{proof}

\begin{proof}[Proof of Theorem~\ref{thm:fair}]
Recall that $\text{Cov}_F(a,\{y>\tilde{q}_\tau\})=\overline{W}_{A-\mathbb{E}A}(\hat{\mu}_\tau,\hat{\nu}_\tau)$. Therefore
\begin{align}
\nonumber& \sup_\tau\left|\text{Cov}_F\left(a,\{y>\tilde{q}_\tau(a,x)\}\right)\right|\\
\nonumber= & \sup_\tau\left|\overline{W}_{A-\mathbb{E}A}(\hat{\mu}_\tau,\hat{\nu}_\tau)\right|\\
\label{eq:triangle}\leq & \sup_{\mu,\nu\in\mathbb{R}}\left|\left(W_{A-\mathbb{E}A}(\mu,\nu)-\overline{W}_{A-\mathbb{E}A}(\mu,\nu)\right)\right|+ \sup_\tau\left|W_{A-\mathbb{E}A}\left(\hat{\mu}_\tau,\hat{\nu}_\tau\right)\right|.
\end{align}
Use Lemma~\ref{lmm:emp} to control the tail of the first term. Apply Lemma~\ref{lmm:emp} with
\[
\mathcal{F}=\left\{f: (a,r)\mapsto (a-\mathbb{E}a)\mathds{1}\{r>\mu a+\nu\}: \mu,\nu\in\mathbb{Q}\right\}.
\]
Note that we are only allowing $\mu,\nu$ to take rational values because Lemma~\ref{lmm:emp} only applies to countable sets of functions. This restriction will not hurt us because the supremum of the $W$ processes over all $\mu,\nu\in\mathbb{R}$ equals the supremum over all $\mu,\nu\in\mathbb{Q}$. Let $F(a,r)=|a-\mathbb{E}a|$ be the envelope function. We need to check that $\text{Subgraph}(\mathcal{F})$ is a VC class of sets.
\begin{align}
\nonumber\text{Subgraph}(\mathcal{F}) =& \left\{\{(a,r,t): (a-\mathbb{E}a)\mathds{1}\{r>\mu a+\nu\}\leq t\} : \mu,\nu\in\mathbb{R}\right\}\\
\label{eq:vc.check}= & \left\{\{(a,r,t): \left(\{r>\mu a+\nu\}\cap \{a-\mathbb{E}a\leq t\}\right)\cup \left(\{r\leq \mu a+\nu\}\cap \{t\geq 0\}\right)\} : \mu,\nu\in\mathbb{Q}\right\}.
\end{align}
Since half spaces in $\mathbb{R}^2$ are of VC dimension 3~\cite[p~221]{alon2004probabilistic}, the set $\{\{r\leq\mu a+\nu\}:\mu,\nu\in\mathbb{Q}\}$ forms a VC class. By the same arguments all four events in~\eqref{eq:vc.check} form VC classes. Deduce that $\text{Subgraph}(\mathcal{F})$ is also a VC class because the VC property is stable under any finitely many union/intersection operations. The assumptions of Lemma~\ref{lmm:emp} are satisfied, which gives that for all $t\geq 2C_1/\sqrt{n}$,
\begin{align*}
&\mathbb{P}\left\{\sup_{\mu,\nu\in\mathbb{R}}\left|W_{A-\mathbb{E}A}(\mu,\nu)-\overline{W}_{A-\mathbb{E}A}(\mu,\nu)\right|>\frac{t}{2}\right\}\\
\leq & 4\mathbb{E}\exp\left(-\left(\frac{nt}{2C_2|A-\mathbb{E}A|}-1\right)^2\right)+4\mathbb{P}\left\{2|A-\mathbb{E}A|>nt/2C_2\right\}\\
\leq & 4\exp\left(-\left(\frac{\sqrt{n}t}{2C_2 u}-1\right)^2\right)+4\mathbb{P}\left\{|A-\mathbb{E}A|>\sqrt{n}u\right\}+4\mathbb{P}\left\{2|A-\mathbb{E}A|>nt/2C_2\right\},\;\;\;\forall u>0.
\end{align*}
Here $|\cdot|$ denotes the Euclidean norm in $\mathbb{R}^n$. Under the assumption that $A_i$ has finite second moment, we could pick $u$ to be a large enough constant, and pick $t$ to be a large enough constant multiple of $1/\sqrt{n}$ to make the above arbitrarily small. In other words,
\[
\sup_{\mu,\nu\in\mathbb{R}}\left|W_{A-\mathbb{E}A}(\mu,\nu)-\overline{W}_{A-\mathbb{E}A}(\mu,\nu)\right|=O_p\left(\frac{1}{\sqrt{n}}\right).
\]
Under the stronger assumption that $A_i-\mathbb{E}A_i$ is sub-Gaussian, we have that $(A_i-\mathbb{E}A_i)^2-\text{Var}(A_i)$ is sub-exponential. Choose $u$ to be a large enough constant and we have
\[
\mathbb{P}\left\{|A-\mathbb{E}A|>\sqrt{n}u\right\}= \mathbb{P}\left\{\frac{1}{\sqrt{n}}\sum_{i\leq n}(A_i-\mathbb{E}A_i)^2>\sqrt{n}u^2\right\}\leq \exp(-C_4(\sqrt{n}u^2-1)).
\]
Similarly if $t>C_1/\sqrt{n}$ for some large enough constant $C_1$, there exists $C_5>0$ such that
\[
\mathbb{P}\left\{2|A-\mathbb{E}A|>nt/2C_2\right\}\leq \exp\left(-C_5(nt^2-1)\right).
\]
Organizing all the terms yields for some positive constants $C,C_1,C_2,C_3$ whose values may have changed from previous lines,
\[
\mathbb{P}\left\{\sup_{\mu,\nu\in\mathbb{R}}\left|W_{A-\mathbb{E}A}(\mu,\nu)-\overline{W}_{A-\mathbb{E}A}(\mu,\nu)\right|>\frac{t}{2}\right\}\leq C\left(\exp\left(-C_2 nt^2\right)+\exp\left(-C_3 \sqrt{n}\right)\right),\;\;\;\forall t>C_1/\sqrt{n}.
\]
For the second term of~\eqref{eq:triangle}, write
\[
W_{A-\mathbb{E}A}\left(\hat{\mu}_\tau,\hat{\nu}_\tau\right) = \frac{1}{n}\sum_{i\leq n}\left(A_i-\mathbb{E}A_i\right)\left\{R_i>\hat{\mu}_\tau A_i+\hat{\nu}_\tau\right\}.
\]
By the dual form of quantile regression~\cite[p~308]{gutenbrunner1992regression}, there exists regression rank scores $b_\tau\in [0,1]^n$ such that
\[
A^T b=(1-\tau)A^T\mathbbm{1},\;\;\; \mathbbm{1}^T b=(1-\tau)n,\;\;\;\text{and}
\]
\[
b_{\tau,i} =\mathbbm{1}\{R_i>\hat{\mu}_\tau A_i+\hat{\nu}_\tau\},\;\;\;\forall i\notin M_\tau,
\]
for some $M_\tau\subset[n]$ of size at most $p$. As a result,
\begin{align*}
& \sup_\tau\left|W_{A-\mathbb{E}A}\left(\hat{\mu}_\tau,\hat{\nu}_\tau\right)\right|\\
\leq & \frac{1}{n}\sup_\tau \left|A^T b_\tau-\mathbb{E}A_1 \frac{1}{n}\mathbbm{1}^T b_\tau\right|+\frac{1}{n}\sup_\tau\left|\sum_{i\in M_\tau}(A_i-\mathbb{E}A_i)\left(b_{\tau,i}-\{R_i>\hat{\mu}_\tau A_i+\hat{\nu}_\tau\}\right)\right|\\
= & \frac{1}{n}\sup_\tau\left|(1-\tau)A^T\mathbbm{1}-\mathbb{E}A_1 (1-\tau)n\right|+\frac{1}{n}\sup_\tau\left|\sum_{i\in M_\tau}(A_i-\mathbb{E}A_i)\left(b_{\tau,i}-\{R_i>\hat{\mu}_\tau A_i+\hat{\nu}_\tau\}\right)\right|\\
\leq & \left|\frac{1}{n}\sum_{i\leq n}(A_i-\mathbb{E}A_i)\right|+\frac{p}{n}\max_{i\leq n}|A_i-\mathbb{E}A_i|.
\end{align*}

If $A_i$ has finite second moment, the above is clearly of order $O_p(1/\sqrt{n})$. If we have in addition that $A_i-\mathbb{E}A_i\sim \mbox{SubG}(\sigma)$, then $|\sum_i (A_i-\mathbb{E}A_i)/n|\sim \mbox{SubG}(\sigma/\sqrt{n})$. For all $t>0$,
\[
\mathbb{P}\left\{\left|\frac{1}{n}\sum_{i\leq n}(A_i-\mathbb{E}A_i)\right|>\frac{t}{4}\right\}\leq 2\exp\left(-\frac{nt^2}{32\sigma^2}\right).
\]
We also have
\[
\mathbb{P}\left\{\frac{p}{n}\max_{i\leq n}|A_i-\mathbb{E}A_i|>\frac{t}{4}\right\}\leq n\mathbb{P}\left\{|A_1-\mathbb{E}A_1|>\frac{tn}{4p}\right\}\leq 2\exp\left(-\frac{n^2t^2}{32\sigma^2p^2}+\log n\right).
\]
Hence
\begin{align*}
&\mathbb{P}\left\{\sup_\tau\left|\text{Cov}_F\left(a,\{y>\tilde{q}_\tau(a,x)\}\right)\right|>t\right\}\\
\leq & \mathbb{P}\left\{\sup_{\mu,\nu\in\mathbb{R}}\left|\left(W_{A-\mathbb{E}A}(\mu,\nu)-\overline{W}_{A-\mathbb{E}A}(\mu,\nu)\right)\right|>\frac{t}{2}\right\}\\
&  + \mathbb{P}\left\{\left|\frac{1}{n}\sum_{i\leq n}(A_i-\mathbb{E}A_i)\right|>\frac{t}{4}\right\} + \mathbb{P}\left\{\frac{p}{n}\max_{i\leq n}|A_i-\mathbb{E}A_i|>\frac{t}{4}\right\}\\
\leq &C\left(\exp\left(-C_2 nt^2\right)+\exp\left(-C_3 \sqrt{n}\right)\right)+ 2\exp\left(-\frac{nt^2}{32\sigma^2}\right)+2\exp\left(-\frac{n^2t^2}{32\sigma^2p^2}+\log n\right)\\
\leq & C\left(\exp\left(-C'_2 nt^2\right)+\exp\left(-C_3 \sqrt{n}\right)+n\exp\left(-C_4 n^2t^2\right)\right).
\end{align*}
That concludes the proof of~\eqref{eq:fair}. The proof of~\eqref{eq:faithful} is similar. Simply note that
\begin{align*}
& \sup_\tau\left|E_F\{y>\tilde{q}_\tau(a,x)\}-(1-\tau)\right|\\
= & \sup_\tau \left|\overline{W}_{\mathds{1}}(\hat{\mu}_\tau,\hat{\nu}_\tau)-(1-\tau)\right|\\
\leq & \sup_{\mu,\nu}\left|W_{\mathds{1}}(\mu,\nu)-\overline{W}_{\mathds{1}}(\mu,\nu)\right|+ \sup_\tau \left|\frac{1}{n}\sum_{i\in M_\tau}\left(b_{\tau,i}-\mathds{1}\{R_i>\hat{\mu}_\tau A_i+\hat{\nu}_\tau\}\right)\right|\\
\leq & \sup_{\mu,\nu}\left|W_{\mathds{1}}(\mu,\nu)-\overline{W}_{\mathds{1}}(\mu,\nu)\right|+\frac{p}{n}
\end{align*}
because $|M_\tau|\leq p$. Apply Lemma~\ref{lmm:emp} with
\[
\mathcal{F}=\{f:(a,r)\mapsto \mathds{1}\{r>\mu a+\nu\}:\mu,\nu\in\mathbb{Q}\},\;\;\;\text{and }F\equiv 1.
\]
The subgraph of $\mathcal{F}$ also forms a VC set via similar analysis. Lemma~\ref{lmm:emp} implies that if $t\geq C_1/\sqrt{n}$ for large enough $C_1$
\[
\mathbb{P}\left\{\sup_{\mu,\nu}\left|W_{\mathds{1}}(\mu,\nu)-\overline{W}_{\mathds{1}}(\mu,\nu)\right|>t\right\}\leq 4\exp\left(-\left(\frac{\sqrt{n}t}{C_2}-1\right)^2\right)+\mathbb{P}\left\{2>\frac{C_1}{C_2}\right\}.
\]
The second term is 0 for $C_1>2C_2$, and the desired inequality~\eqref{eq:faithful} immediately follows.
\end{proof}

\begin{proof}[Proof of Theorem~\ref{thm:risk}]
Suppose $\mu_\tau^*, \nu_\tau^*\in \arg\min_{\mu,\nu\in\mathbb{R}}\mathcal{R}(\hat{q}_\tau+\mu A + \nu)$. There exists some finite constant $K$ for which
\[
(\mu_\tau^*, \nu_\tau^*)\in B_K=\{(\mu,\nu): \max (|\mu|,|\nu|)\leq K\}.
\]
Invoke Lemma~\ref{lmm:emp} with
\[
\mathcal{F}=\left\{f:(a,r)\mapsto \rho_\tau(r-\mu a-\nu): \mu,\nu\in\mathbb{Q}\right\}.
\]
The subgraph of $\mathcal{F}$ forms a VC class of sets, and on the compact set $B_{K}$, we have $|f|\leq F$ where $F(a,r)=|r|+K|a|+K$ has bounded second moment. By Lemma~\ref{lmm:emp},
\begin{equation}\label{eq:R.concentration}
\sup_{(\mu,\nu)\in B_{2K}}\left|\frac{1}{n}\sum_{i\leq n}\left(\rho_\tau(R_i-\mu A_i-\nu)-\mathbb{E}\rho_\tau (R_i-\mu A_i-\nu)\right)\right|=O_p(1/\sqrt{n}).
\end{equation}
Use continuity of $\rho_\tau$ to deduce existence of some $\delta>0$ for which
\[
\mathbb{E}\rho_\tau(R_1-\mu A_1 -\nu)>\mathbb{E}\rho_\tau(R_1-\mu^*_\tau A_1 -\nu^*_\tau)+2\delta\;\;\;\forall (\mu,\nu)\in\partial B_{2K}.
\]
Use~\eqref{eq:R.concentration} to deduce that with probability $1-o(1)$,
\[
\min_{(\mu,\nu)\in \partial B_{2K}}\frac{1}{n}\sum_{i\leq n}\rho_\tau (R_i-\mu A_i-\nu)>\mathbb{E}\rho_\tau(R_1-\mu^* A_1 -\nu^*)+\delta> \frac{1}{n}\sum_{i\leq n}\rho_\tau (R_i-\mu_\tau^* A_i-\nu_\tau^*).
\]
By convexity of $\rho_\tau$, the minimizers $\hat{\mu}_\tau, \hat{\nu}_\tau$ must appear with $B_{2K}$. Recall that $\hat{\mu}_\tau$ and $\hat{\nu}_\tau$ are obtained by running quantile regression of $R$ against $A$ on the training set, so we have
\begin{equation}\label{eq:basic}
\frac{1}{n}\sum_{i\leq n}\rho_\tau(R_i-\hat{\mu}_\tau A_i-\hat{\nu}_\tau)\leq \frac{1}{n}\sum_{i\leq n}\rho_\tau(R_i-\mu_\tau^* A_i-\nu_\tau^*).
\end{equation}
A few applications of the triangle inequality yields
\begin{align*}
& R(\tilde{q}_\tau)-R(\tilde{q}_\tau^*)\\
\leq & \frac{1}{n}\sum_{i\leq n}\rho_\tau(R_i-\hat{\mu}_\tau A_i-\hat{\nu}_\tau)
-\frac{1}{n}\sum_{i\leq n}\rho_\tau(R_i-\mu_\tau^* A_i-\nu_\tau^*)\\
& + 2\sup_{(\mu,\nu)\in B_{2K}}\left|\frac{1}{n}\sum_{i\leq n}\left(\rho_\tau(R_i-\mu A_i-\nu)-\mathbb{E}\rho_\tau (R_i-\mu A_i-\nu)\right)\right|\\
\leq & 0+O_p(1/\sqrt{n})
\end{align*}
by~\eqref{eq:basic} and~\eqref{eq:R.concentration} .
\end{proof}

%% file: ack.tex
\section*{Acknowledgment}
Research supported in part by ONR grant N00014-12-1-0762 and NSF grant DMS-1513594.